\documentclass[a4paper,11pt,oneside]{article}

\usepackage[utf8]{inputenc}
\usepackage[english]{babel}
\usepackage[T1]{fontenc}           	
\usepackage{amsmath}
\usepackage{amssymb}
\usepackage{amsfonts}          		
\usepackage[all]{xy}
\usepackage{amscd}
\usepackage{amsthm}			

\usepackage{color}
\usepackage{psfrag}
\usepackage[nottoc]{tocbibind}
\usepackage{framed,color}
\usepackage[new]{old-arrows}
\usepackage[]{listings}
\usepackage[]{float}
\usepackage{hyperref}
\usepackage[dvips,final]{graphicx}
\usepackage[dvips]{geometry}
\usepackage{epsfig}
\usepackage{latexsym}

\usepackage{bm}
\usepackage{comment}
\usepackage{xfrac}
\usepackage{microtype}
\usepackage{epstopdf}
\usepackage{scalerel}
\usepackage{lmodern}
\usepackage{textcomp}
\usepackage{tipa}

\usepackage{enumitem}

\usepackage[all]{pstricks}
\usepackage{caption}
\usepackage{subcaption}
\usepackage{mathtools}
\DeclarePairedDelimiter\ceil{\lceil}{\rceil}
\DeclarePairedDelimiter\floor{\lfloor}{\rfloor}

\newcounter{count}[section]

\theoremstyle{plain}
\newtheorem{theorem}[count]{Theorem}
\newtheorem{prop}[count]{Proposition}
\newtheorem{cor}[count]{Corollary}
\newtheorem{lemma}[count]{Lemma}

\newtheoremstyle{named}{}{}{\itshape}{}{\bfseries}{.}{ }{#1 \thmnote{#3}}
\theoremstyle{named}
\newtheorem*{namedtheorem}{Theorem}

\theoremstyle{definition}
\newtheorem{definition}[count]{Definition}
\newtheorem{ex}[count]{Example}
\newtheorem{remark}[count]{Remark}

\newtheorem{conj}{Conjecture}

\newcommand{\Z}{\mathbb{Z}}
\newcommand{\Q}{\mathbb{Q}}
\newcommand{\R}{\mathbb{R}}

\newcommand{\D}{\mathbb{D}}

\begin{document}
\title{L-spaces, taut foliations and the Whitehead link}
\author{Diego Santoro}
\date{}

\maketitle

\begin{abstract}

    We prove that if $M$ is a rational homology sphere that is a Dehn surgery on the Whitehead link, then $M$ is not an $L$-space if and only if $M$ supports a coorientable taut foliation. The left orderability of some of these manifolds is also proved, by determining which of the constructed taut foliations have vanishing Euler class.
    
    We also present some more general results about the structure of the $L$-space surgery slopes for links with two unknotted components and linking number zero, and about the existence of taut foliations on the fillings of a $k$-holed torus bundle over the circle with some prescribed monodromy. Our results, combined with some results from \cite{RSS}, also imply that all the rational homology spheres that arise as integer surgeries on the Whitehead link satisfy the L-space conjecture. 
    
\end{abstract}

\section{Introduction}
In this paper we study the rational homology spheres obtained as surgery on the Whitehead link, motivated by the following conjecture.

\begin{conj}[L-space conjecture]\label{L space conjecture} 
For an irreducible oriented rational homology $3$-sphere $M$, the following are
equivalent:
\begin{enumerate}[label=\arabic*)]
    \item  $M$ supports a cooriented taut foliation;
    \item  $M$ is not an L-space, i.e. its Heegaard Floer homology is not minimal;
    \item  $M$ is left orderable, i.e. $\pi_1(M)$ is left orderable.
\end{enumerate}
\end{conj}
The equivalence between $1)$ and $2)$ was conjectured by Juh\'asz in \cite{J}, while the equivalence between $2)$ and $3)$ was conjectured by Boyer, Gordon and Watson in \cite{BGW}.

Even if the properties involved in the conjecture are very different in flavour and nature it follows by the works of Oszv\'ath-Szab\'o \cite{OS}, Bowden  \cite{Bow} and Kazez-Roberts \cite{KR2} that $1)$ implies $2)$. Moreover it is now known that the conjecture holds for all the graph manifolds (\cite{RS,BC,HRRW}). 

It is therefore interesting to study the conjecture in the case of hyperbolic manifolds. In this direction, in \cite{Z} the conjecture is proved for some manifolds obtained by considering mapping tori of pseudo-Anosov diffeomorphisms of closed surfaces and then by surgering on some collections of closed orbits. 

In addition, in \cite{D}, the conjecture is tested on a census of more than $300,000$ hyperbolic rational homology spheres and proved for more than $60\%$ of these manifolds.
\\

One way of producing rational homology spheres is via Dehn surgery on knots or links in $S^3$. When dealing with surgeries on knots, the different aspects of this conjecture have been studied separately in several papers.
For example it has been proved that if a knot $K$ admits a positive surgery that is an L-space then $K$ is fibered \cite{G, Ni}, strongly quasipositive \cite{Hedden} and the $r$-framed surgery along $K$ is an L-space if and only if $r\geq 2g(K)-1$, where $g(K)$ denotes the genus of the knot $K$ \cite{KMOS}. 

Taut foliations on surgeries on knots are constructed, for example, in \cite{R1}, \cite{LR}, \cite{DR1}, \cite{DR2}, \cite{K} and it is possible to prove the left orderability of some of these manifolds by determining which of these foliations have vanishing Euler class, as done in \cite{H}. Another approach to study the left orderability of surgeries on knots is via representation theoretic methods, as presented in \cite{CuD}.
\\

On the other hand, not much is known when it comes to the study of surgeries on links. Some results regarding integer L-space surgeries on links in $S^3$ are presented in \cite{GN}, \cite{GH}, \cite{GN1}, \cite{Liu} and \cite{GLM}, while in \cite{Rasmussen} rational L-space surgeries on satellites by algebraic links are studied. 
Concerning foliations, in \cite{KR1} Kalelkar and Roberts construct coorientable taut foliations on some fillings of $3$-manifolds that fiber over the circle. In particular, their methods can also be applied to surgeries on fibered links.
\\

As far as we know, in this paper we provide the first example of the equivalence between conditions $1)$ and $2)$ of the conjecture for manifolds obtained via Dehn surgery on a hyperbolic link with at least $2$ components.

We focus our attention on the Whitehead link, that is depicted in Figure \ref{Whitehead link}.

\begin{figure}[H]
    \centering
    \includegraphics[width=0.35\textwidth]{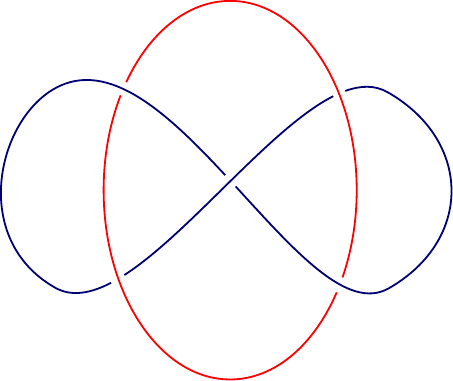}
    \caption{The Whitehead link}
    \label{Whitehead link}
\end{figure}

We will denote the Whitehead link with {\rm WL} and the $\left(\sfrac{p_1}{q_1}, \sfrac{p_2}{q_2}\right)$-surgery on the Whitehead link with $S^3_{\sfrac{p_1}{q_1}, \sfrac{p_2}{q_2}}({\rm WL})$. Notice that since the Whitehead link has linking number zero, the homology of this manifold is isomorphic to $\Z_{p_1}\oplus \Z_{p_2}$. In particular, the manifold $S^3_{\sfrac{p_1}{q_1}, \sfrac{p_2}{q_2}}({\rm WL})$ is a rational homology sphere if and only if $p_1\ne 0$ and $p_2\ne 0$.

Recall that the Whitehead link exterior supports a complete hyperbolic structure and therefore, by virtue of Thurston's Hyperbolic Dehn surgery theorem \cite{Th1}, ``most'' of its fillings are hyperbolic. 

The main result of this paper is the following:

\begin{theorem}\label{theorem}
Let $p_1, q_1$ and $p_2,q_2$ be two pairs of non vanishing coprime integers.
Let $S^3_{\sfrac{p_1}{q_1}, \sfrac{p_2}{q_2}}({\rm WL})$ be the $\left(\sfrac{p_1}{q_1}, \sfrac{p_2}{q_2}\right)$-surgery on the Whitehead link. Then
\begin{itemize}
    \item $S^3_{\sfrac{p_1}{q_1}, \sfrac{p_2}{q_2}}({\rm WL})$ is an L-space if and only if $\sfrac{p_1}{q_1}\geq 1$ and $\sfrac{p_2}{q_2}\geq 1$.
    \item $S^3_{\sfrac{p_1}{q_1}, \sfrac{p_2}{q_2}}({\rm WL})$ supports a cooriented taut foliation if and only if $\sfrac{p_1}{q_1}< 1$ or $\sfrac{p_2}{q_2}< 1$.
\end{itemize}
In particular, for all these manifolds  we have $1)\Leftrightarrow 2)$ in Conjecture \ref{L space conjecture}.
\end{theorem}

Recall that since the Whitehead link has linking number zero, if $p_1=0$ or $p_2=0$ then $S^3_{\sfrac{p_1}{q_1}, \sfrac{p_2}{q_2}}({\rm WL})$ is not a rational homology sphere. This implies that when some of the parameters $p_1,p_2,q_1,q_2$ vanish the only rational homology spheres that can be obtained are $S^3$ and lens spaces and it is known that they satisfy the L-space conjecture.  

Also notice that the statement of Theorem \ref{theorem} is invariant under the switch $\sfrac{p_1}{q_1}\leftrightarrow \sfrac{p_2}{q_2}$. This is a consequence of the fact that the Whitehead link is symmetric, i.e. that there exists an isotopy exchanging its two components. 
\\

In the proof of the theorem we study the conditions of being an L-space and of supporting a taut foliation separately.

The key idea in the proof of the first part of the theorem is to use the results of J. Rasmussen and S.D. Rasmussen in \cite{RR} and Gorsky, Liu and Moore in \cite{GLM} to prove the following more general fact:

\begin{theorem}\label{struttura L space surgery}
Suppose that $\mathcal{L}$ is a non-trivial link in $S^3$ with two unknotted components and linking number zero. Suppose that there exist rationals $\sfrac{p_1}{q_1}> 0$ and $\sfrac{p_2}{q_2}>0$ such that $S^3_{\sfrac{p_1}{q_1},\sfrac{p_2}{q_2}}(\mathcal{L})$ is an L-space. Then there exist non-negative integer numbers $b_1,b_2$ such that 
$$
L(\mathcal{L})=\big{(}[2b_1+1,\infty]\times[2b_2+1,\infty]\big{)}\cup \big{(}\{\infty\}\times \Q^*\big{)} \cup  \big{(}\Q^*\times\{\infty\}\big{)}.
$$
\end{theorem}

In the previous statement the symbol $L(\mathcal{L})$ denotes the set of the \emph{L-space filling slopes} of the exterior of $\mathcal{L}$. This is the set of slopes such that filling the exterior of $\mathcal{L}$ with such slopes yields L-spaces.

\begin{remark}
We will not make use of this fact, but the integers $b_1$ and $b_2$ can be explicitely computed with the $H$ function associated to $\mathcal{L}$, see \cite{GLM}. 
\end{remark}

The foliations, on the other hand, are obtained by constructing branched surfaces without sink discs, using the results of Li (\cite{L},\cite{L2}) and inspired by the works of Roberts and Kalelkar-Roberts (\cite{R},\cite{R1},\cite{KR1}).
Also in this case, part of the proof of Theorem \ref{theorem} will follow from a more general result, that allows to construct taut foliations on some fillings of manifolds that fiber over the circle with fiber a $k$-holed torus and some particular type of monodromy. This is the content of Theorem \ref{teorema foliazioni}, which seems to be interesting in itself. 
\newline

We are also able to determine which of the taut foliations of Theorem \ref{theorem} have zero Euler class, by adapting the ideas of Hu in \cite{H} to our case. This implies that the manifolds supporting such taut foliations have left orderable first fundamental group.

More precisely we have:
\begin{theorem}\label{theorem euler}
Let $S^3_{\sfrac{p_1}{q_1}, \sfrac{p_2}{q_2}}({\rm WL})$ be the $\left(\sfrac{p_1}{q_1}, \sfrac{p_2}{q_2}\right)$-surgery on the Whitehead link, with $q_1,q_2\ne 0$ and $p_1,p_2>0$.

Then the foliations constructed in the proof of the Theorem \ref{theorem} have vanishing Euler class if and only if $|q_i|\equiv 1 \pmod{p_i}$ for $i=1,2$. 

In particular, for all these manifolds the L-space conjecture holds.
\end{theorem}

Moveover, Theorem \ref{theorem}, together with some results proved in \cite{RSS}, implies the following

\begin{theorem}\label{Z+RS}
All the rational homology spheres obtained by integer surgery on {\rm WL} satisfy the L-space conjecture.
\end{theorem}
We refer to the last section of this paper for a more detailed statement of Theorem \ref{Z+RS}, which also combines results from \cite{Z}. 
\\

\textbf{Structure of the paper.} The paper is organised as follows. In Section $2$ we recall the result of \cite{RR} and we prove Theorem \ref{struttura L space surgery} and the first part of Theorem \ref{theorem}.
In Section $3$ we focus our attention on taut foliations. In Subsection $3.1$ we recall some basic notions on branched surfaces and the main result of \cite{L2}. In Subsection $3.2$ we prove Theorem \ref{teorema foliazioni} and start the proof of the second part of Theorem \ref{theorem}, that is concluded in Subsection $3.3$. In the last section we prove Theorem \ref{theorem euler} and collect from \cite{D}, \cite{RSS} and \cite{Z} some other results about orderability, and non-orderability, of some surgeries on the Whitehead link.
\\

\textbf{Acknowledgments.} I warmly thank my advisors Bruno Martelli and Paolo Lisca for having presented this problem to me, for their support and for their useful comments on this paper. I also thank Alice Merz and Ludovico Battista for the several helpful discussions. I also thank the referee for the many valuable comments and suggestions.

\section{L-spaces}

In this section we prove the first part of Theorem \ref{theorem}. 
We start by recalling some definitions and the main result of \cite{RR}.
Let $Y$ be a rational homology solid torus, i.e. $Y$ is a compact oriented $3$-manifold with toroidal boundary such that $H_*(Y;\mathbb{Q})\cong H_*(\D^2\times S^1; \Q)$.

We are interested in the study of the Dehn fillings on $Y$. We define the \textbf{set of slopes in Y} as
$$
Sl(Y)=\{\alpha\in H_1(\partial Y;\Z)|\, \alpha \text{ is primitive}\}/\pm 1.
$$
It is a well known fact that each element  $[\alpha] \in Sl(Y)$ determines a Dehn filling on $Y$, that we will denote with $Y(\alpha)$. 

Notice that since $Y$ is a rational homology solid torus, there is a distinguished slope in $Sl(Y)$ that we call the \emph{homological longitude} of $Y$ and that is defined in the following way. We denote with $i: H_1(\partial Y;\Z)\rightarrow H_1(Y;\Z)$ the map induced by the inclusion $\partial Y\subset Y$ and we consider a primitive element $l\in H_1(\partial Y;\Z)$ such that $i(l)$ is torsion in $H_1(Y;\Z)$. The element $l$ is unique up to sign, and its equivalence class $[l]\in Sl(Y)$ is the homological longitude of $Y$. This definition, that may seem to be counterintuitive, is given so that when $Y$ is the complement of a knot in $S^3$, the homological longitude of $Y$ coincides with the slope defined by the longitude the knot.
\newline

We want to study the fillings on $Y$ that are L-spaces. 
For this reason we define the set of the \textbf{L-space filling slopes}:
$$
L(Y)=\{[\alpha]\in Sl(Y)|\,\, Y(\alpha)\text{ is an L-space}\}
$$

and we say that $Y$ is \textbf{Floer simple} if $Y$ admits multiple L-space filling slopes, i.e. if $|L(Y)|>1$.
\newline

It turns out that if $Y$ is Floer simple then the set $L(Y)$ has a simple structure, and this can be computed by knowing the \textbf{Turaev torsion} of $Y$. We only recall some properties of the Turaev torsion and we refer the reader to \cite{T} for the precise definitions. 

Fix an identification $H_1(Y;\Z)=\Z \oplus T$, where $T$ is the torsion subgroup, and denote by $\phi: H_1(Y;\Z)\rightarrow \Z$ the projection induced by this identification. Then the Turaev torsion of $Y$ can be normalised to be written as a formal sum
$$
\tau(Y)=\sum_{\substack{h\in H_1(Y;\Z) \\ \phi(h)\geq 0}}a_h h
$$
where $a_h$ is an integer for each $h$, $a_0\ne 0$ and $a_h=1$ for all but finitely many $h$.
For example (see \cite[Section II.5]{T}) if $H_1(Y;\Z)=\Z$ the Turaev torsion of $Y$ can be written as
$$
\tau(Y)=\frac{\Delta(Y)}{1-t}\in \Z[[t]]
$$
 where $(1-t)^{-1}$ is expanded as an infinite sum in positive powers of $t$ and $\Delta(Y)$ is the Alexander polynomial of $Y$  normalised so that $\Delta(Y)\in \Z[t]$, $\Delta(Y)(0)\ne 0$ and $\Delta(Y)(1)= 1$.
 In fact, in this case the coefficients of $\tau(Y)$ are eventually constant and equal to the sum of all the coefficients of $\Delta(Y)$, and this value is exactly $\Delta(Y)(1)=1$.
\\

We define $S[\tau(Y)]=\{h\in H_1(Y;\Z)|\,\, a_h\ne 0\}$ to be the \textbf{support of $\tau(Y)$}.

We also define the following subset of $H_1(Y;\Z)$:
$$
D_{>0}^\tau(Y)=\big{\{}x-y\,|\, x\notin S[\tau(Y)], y\in S[\tau(Y)]\text{ and } \phi(x)>\phi(y) \big{\}}\cap i(H_1(\partial Y; \Z))
$$
where $i:H_1(\partial Y;\Z)\rightarrow H_1(Y;\Z)$ is induced from the inclusion.
\begin{lemma}\label{D_tau è finito}
 The set $D^{\tau}_{>0}$ is always finite.
\begin{proof} 
Recall that we have fixed an identification $H_1(Y, \Z)=\Z\oplus T$, where $T$ is the torsion subgroup, and we have denoted by $\phi:H_1(Y,\Z)\rightarrow \Z$ the projection induced by this identification. Also recall that the Turaev torsion of $Y$ is normalised so to be written as 
$$
\tau(Y)=\sum_{\substack{h\in H_1(Y;\Z) \\ \phi(h)\geq 0}}a_h h
$$
where $a_h$ is an integer for each $h$, $a_0\ne 0$ and $a_h=1$ for all but finitely many $h$. This implies in particular that if $h\in S[\tau(Y)]$ then $\phi(h)\geq 0$. Moreover since $a_h=1$ for all but finitely many $h$ with $\phi(h)\geq 0$ we also deduce that there exists a positive constant $c\in \Z$ such that if $h'\notin S[\tau(Y)]$ then $\phi(h')\leq c$.

We now prove that $D^{\tau}_{>0}$ is finite. To do this, we define for each $x\notin S[\tau(Y)]$ the set
$$
\mathcal{S}^{\tau}_x=\{y\in S[\tau(Y)]\,|\, \phi(x)>\phi(y)\}.
$$
We show that $\mathcal{S}_x^{\tau}$ is always finite and that it is non-empty only for finitely many $x\notin S[\tau(Y)]$. It follows from the definition of $D_{>0}^{\tau}(Y)$ that this implies that $D_{>0}^{\tau}(Y)$ is finite.
We fix $x\notin S[\tau(Y)]$ and we have two cases:
 \begin{itemize}
     \item \emph{$\phi(x)\leq 0$:} in this case, since all the $y\in S[\tau(Y)]$ satisfy $\phi(y)\geq 0$, we have that $\mathcal{S}_x^{\tau}$ is empty.
     \item \emph{$\phi(x)>0$:} we use again the fact that all the $y\in S[\tau(Y)]$ satisfy $\phi(y)\geq 0$ to deduce that
     $$
     \mathcal{S}_x^{\tau}\subset \{0,1,\dots, \phi(x)-1\}\oplus T\subset \Z\oplus T=H_1(Y,\Z).
     $$
     Since the torsion subgroup $T$ is finite we have that $S^\tau_x$ is finite.
 \end{itemize}
 To conclude the proof we show that the latter case occurs only for finitely many $x\notin S[\tau(Y)]$. In fact since there exists a positive constant $c\in \Z$ such that if $x\notin S[\tau(Y)]$ then $\phi(x)\leq c$ we have that the set $\{x\notin S[\tau(Y)]\,|\,\phi(x)>0\}$ is contained in $\{0,1,\dots, c\}\oplus T$, and this is a finite set.
 \end{proof}
\end{lemma}

We are now ready to state the main result of \cite{RR} :

\begin{theorem}(\cite{RR}) \label{theorem RR}
If $Y$ is Floer simple, then either
\begin{itemize}
    \item $D_{>0}^\tau(Y)=\emptyset$ and $L(Y)=Sl(Y)\setminus [l]$, or
    \item $D_{>0}^\tau(Y)\ne\emptyset$ and $L(Y)$ is a closed interval whose endpoints are consecutive elements in $i^{-1}(D_{>0}^\tau(Y))$.
\end{itemize}
\end{theorem}
We explain more precisely the second part of the statement of this theorem. Once we fix a basis $(\mu,\lambda)$ for $H_1(\partial Y; \Z)$ we can associate to each element $a\mu+b\lambda\in H_1(\partial Y; \Z)$ the element $\sfrac{a}{b}\in \overline{\Q}=\Q\cup \{\infty\} \subset S^1$. This association defines a map onto $\overline{\Q}$ that yields an identification between $Sl(Y)$ and $\overline{\Q}$.

If the set $D^\tau_{>0}$ is not empty, then we can apply this map to the set $i^{-1}(D^\tau_{>0})\subset H_1(\partial Y; \Z)$ and Theorem \ref{theorem RR} states that if $Y$ is Floer simple then $L(Y)$ is a closed interval in $Sl(Y)=\overline{\Q}$ whose endpoints are consecutive elements in the image of $i^{-1}(D^\tau_{>0})$ in $\overline{\Q}$.
\newline

In the case of our interest we consider a link $\mathcal{L}=K_1\sqcup K_2\subset S^3$ with the following properties:
\begin{itemize}
    \item $\mathcal{L}$ has two unknotted components;
    \item $\mathcal{L}$ has linking number zero.
\end{itemize}
By analogy with the definitions given for rational homology solid tori we denote with $Sl(\mathcal{L})$ the set of slopes of the exterior of $\mathcal{L}$ and with $L(\mathcal{L})$ the set of L-space filling slopes of the exterior of $\mathcal{L}$.

We fix an orientation of the components of $\mathcal{L}$ and in this way we obtain canonical meridian-longitude bases $(\mu_i,\lambda_i)_{i=1,2}$ of the first homology groups of the boundary tori of its exterior. The choice of these bases also determines an identification $Sl(\mathcal{L})=\overline{\Q}\times \overline{\Q}$.

We denote with $S^3_{\sfrac{p_1}{q_1},\bullet}(\mathcal{L})$ the manifold obtained by drilling the second component of $\mathcal{L}$ and by performing $\left(\sfrac{p_1}{q_1}\right)$-surgery
on the first. Analogously we denote with $S^3_{\bullet,\sfrac{p_2}{q_2}}(\mathcal{L})$ the manifold obtained by drilling the first component of $\mathcal{L}$ and by performing $\left(\sfrac{p_2}{q_2}\right)$-surgery on the second.
When it will not be important which component we are drilling and on which component we are surgering, we will simply use the symbol $S^3_{\sfrac{p}{q}}(\mathcal{L})$.

Notice that since $\mathcal{L}$ has linking number zero, we have an isomorphism 
$$
H_1\left(S^3_{\sfrac{p_1}{q_1},\bullet}(\mathcal{L});\Z\right)\cong\Z_{p_1} \oplus \Z
$$
where the image of the meridian $\mu_1$ in $H_1\Big{(}S^3_{\sfrac{p_1}{q_1},\bullet}(\mathcal{L});\Z\Big{)}$ is mapped to $(1,0)$ and the image of the meridian $\mu_2$ in $H_1\Big{(}S^3_{\sfrac{p_1}{q_1},\bullet}(\mathcal{L});\Z\Big{)}$ is mapped to $(0,1)$. An analogous result holds for $S^3_{\bullet,\sfrac{p_2}{q_2}}(\mathcal{L})$.

\begin{lemma}\label{lemma intervallo}
Fix $p\ne0$ and $q$ coprime integers. Suppose that $S^3_{\sfrac{p}{q}}(\mathcal{L})$ is Floer simple. Then the set $L\left(S^3_{\sfrac{p}{q}}(\mathcal{L})\right)$ has one of the following forms:
\begin{itemize}
    \item $L\left(S^3_{\sfrac{p}{q}}(\mathcal{L})\right)=\overline{\Q}\setminus{\{0\}}$, or
    \item there exists a natural number $k>0$ such that either $L\left(S^3_{\sfrac{p}{q}}(\mathcal{L})\right)=[k,\infty]$ or $L\left(S^3_{\sfrac{p}{q}}(\mathcal{L})\right)=[\infty,-k]$.
\end{itemize}
\end{lemma}

\begin{proof}
We suppose that $S^3_{\sfrac{p}{q}}(\mathcal{L})=S^3_{\bullet,\sfrac{p}{q}}(\mathcal{L})$, the case $S^3_{\sfrac{p}{q}}(\mathcal{L})=S^3_{\sfrac{p}{q},\bullet}(\mathcal{L})$ being analogous. We denote $S^3_{\sfrac{p}{q}}(\mathcal{L})$ with $M$.

The lemma follows from Theorem \ref{theorem RR} together with a simple inspection on the possible forms of the set $D^{\tau}_{>0}(M)$:
\begin{itemize}
\item \emph{$D^{\tau}_{>0}(M)$ is empty:} in this case we have that $L(M)=\overline{\Q}\setminus{\{0\}}$.
    \item \emph{$D^{\tau}_{>0}(M)$ is not empty:} 
    recall that by definition $D_{>0}^{\tau}(M)$ is the subset of $H_1(M;\Z)$ defined as 
$$
D_{>0}^\tau(M)=\big{\{}x-y\,|\, x\notin S[\tau(M)], y\in S[\tau(M)]\text{ and } \phi(x)>\phi(y) \big{\}}\cap i(H_1(\partial M; \Z)).
$$
    In our case the projection $\phi$ associated to the identification 
    $$
    H_1(M;\Z)=\Z\oplus \Z_p
    $$
    is simply the map $\phi(x_1,x_2)=x_1$ and therefore the condition $\phi(x)>\phi(y)$ in the definition of $D_{>0}^{\tau}(M)$ implies that 
    $$
    D_{>0}^{\tau}(M)\subset (\Z_{>0}\times \Z_p) \cap i(H_1(\partial M;\Z)).
    $$ 
    Moreover, since $i(H_1(\partial M; \Z))=\Z\times \{0\}$ we have that $D_{>0}^{\tau}(M)$ is
    a subset of $\Z_{>0}\times \{0\}$, and we denote with $S=\{n_1,\dots, n_h\}\subset \Z_{>0}$ the first coordinates of its elements, listed in ascending order. Recall from Lemma \ref{D_tau è finito} that $D^{\tau}_{>0}(M)$ is always a finite set.
    
    We have that 
    $$
    i^{-1}(D^{\tau}_{>0}(M))=\{(n_i,m)\in\Z\times \Z\,|\,\, n_i\in S\text{ and }m\in \Z\}
    $$
    and we know by Theorem \ref{theorem RR} that $L(M)$ is a closed interval in  $\overline{\Q}$ whose endpoints are consecutive elements in the set $\left\{\sfrac{n_i}{m}|\,\, n_i\in S \text{ and }m\in \Z\right\}$. Since the components of $\mathcal{L}$ are unknotted we know that $S^3_{\infty,\sfrac{p}{q}}(\mathcal{L})$ is a L-space (it is indeed a lens space) and therefore that $\infty$ belongs to $L(M)$. Hence we can conclude that either $L(M)=[n_h,\infty]$ or $L(M)=[\infty, -n_h]$. 
\end{itemize}
This concludes the proof.
\end{proof}
As we already anticipated, the first part of Theorem \ref{theorem} will be a corollary of the following more general result, which we will prove soon:

\begin{namedtheorem} [\ref{struttura L space surgery}]
Suppose that $\mathcal{L}$ is a non-trivial link with two unknotted components and linking number zero. Suppose that there exist rationals $\sfrac{p_1}{q_1}> 0$ and $\sfrac{p_2}{q_2}>0$ such that $S^3_{\sfrac{p_1}{q_1},\sfrac{p_2}{q_2}}(\mathcal{L})$ is an L-space. Then there exist non-negative integer numbers $b_1,b_2$ such that 
$$
L(\mathcal{L})=\big{(}[2b_1+1,\infty]\times[2b_2+1,\infty]\big{)}\cup \big{(}\{\infty\}\times \Q^*\big{)} \cup  \big{(}\Q^*\times\{\infty\}\big{)}.
$$
\end{namedtheorem}

As a corollary of Theorem \ref{struttura L space surgery} we have:

\begin{cor}[First part of Theorem \ref{theorem}]\label{theorem part 1}
The $3$-manifold  $S^3_{\sfrac{p_1}{q_1}, \sfrac{p_2}{q_2}}({\rm WL})$ is an L-space if and only if $\sfrac{p_1}{q_1}\geq 1$ and $\sfrac{p_2}{q_2}\geq 1$. 
\end{cor}

\begin{proof}
The $(1,1)$-surgery on the Whitehead link is the Poincaré homology sphere, that has finite fundamental group and is therefore an L-space. In other words $(1,1)$ belongs to $L({\rm WL})$. The Whitehead link also has unknotted components and linking number zero and we can therefore apply Theorem \ref{struttura L space surgery}, that immediately implies the thesis.
\end{proof}

To prove Theorem \ref{struttura L space surgery} we recall the definition of L-space link, as given by Gorsky and Némethi in \cite{GN}. We give the definition for a $2$-component link, but it is generalisable to links with more components.

\begin{definition}(\cite{GN})
A link $\mathcal{L}\subset S^3$ is an \textbf{L-space link} if all sufficiently large integer surgeries are L-spaces, i.e. if there exist integers $p_1,p_2$ such that $S^3_{d_1,d_2}(\mathcal{L})$ is an L-space for all  integers $d_1>p_1$ and $d_2>p_2$. 
\end{definition}

For knots, the existence of a positive rational L-space surgery implies the existence of arbitrarily large L-space surgeries, but this fails in the case of links, as the Example $2.4$ in \cite{Liu} shows.

Nevertheless, the following lemma shows that such generalisation holds if $\mathcal{L}$ has unknotted components and linking number zero, or more generally if it is a $n$-component \emph{Brunnian} link. Recall that a link $L$ with three or more components is Brunnian if all of its sublinks are unlinks. A link with two components is Brunnian when its components are unknotted and have linking number zero. We will use the symbols $\floor*{x}$ and $\ceil*{x}$, where $x$ is a rational number, to denote the integers
$$
\floor*{x}=\text{max}\{k\in \Z | k\leq x\}
$$
$$
\ceil*{x}=\text{min}\{k\in \Z | k\geq x\}.
$$

\begin{lemma}\label{lemma razionale implica L space}
Let $\mathcal{L}$ be an $n$-component Brunnian link and suppose that there exist rationals $r_1>0,r_2>0,\cdots, r_n>0$ such that $S^3_{r_1,\dots, r_n}(\mathcal{L})$ is an L-space. Then $S^3_{s_1,\dots,s_n}(\mathcal{L})$ is an L-space for
all $(s_1,\dots, s_n)$ satisfying
$$
\begin{cases}
s_i\geq \floor*{r_i}\quad \text{if $r_i\geq1$}\\
s_i>\floor*{r_i}=0\quad \text{if $0<r_i<1$}.
\end{cases}
$$
In particular $\mathcal{L}$ is an L-space link.
\end{lemma}
\begin{proof}
We suppose that $\mathcal{L}$ has two components, the proof being analogous in the general case.
We have the following cases:
\begin{itemize}
    \item \emph{$r_1\geq 1$ and $r_2\geq 1$:}
    we start by considering the set $L(S^3_{r_1,\bullet}(\mathcal{L}))$. We know by hypothesis that this set contains $r_2$ and since $r_2$ is positive it follows from Lemma \ref{lemma intervallo} that $L(S^3_{r_1,\bullet}(\mathcal{L}))$ must be either $\overline{\Q}\setminus\{0\}$ or $[k,\infty]$ for some positive natural number $k$. In both of these cases, since $r_2\geq 1$ we can deduce that 

$$
\left[\floor*{r_2},\infty\right]\subset L(S^3_{r_1,\bullet}(\mathcal{L})).
$$

We now consider the set $L(S^3_{\bullet,\floor*{r_2}}(\mathcal{L}))$. We have just proved that it contains $r_1$ and with the same argument as before we can deduce that 
$$
\left[\floor*{r_1},\infty\right]\subset L(S^3_{\bullet,\floor*{r_2}}(\mathcal{L})).
$$
By applying the same reasoning it follows that for every $s_1\geq \floor*{r_1}$ we have that 
$$
\left[\floor*{r_2},\infty\right]\subset L(S^3_{s_1,\bullet}(\mathcal{L}))
$$
and this is exactly what we wanted. A pictorial sketch of the proof is showed in Figure \ref{proof}. 
    \item \emph{$r_1\geq 1$ and $r_2<1$:} since $0<r_2<1$ we have $L(S^3_{r_1,\bullet}(\mathcal{L}))=\overline{\Q}\setminus\{0\}$ and in particular
    $$
    (0,\infty]=\left(\floor*{r_2},\infty\right]\subset L(S^3_{r_1,\bullet}(\mathcal{L})).
    $$
    This implies that for any $s_2>0$ we have that $S^3_{r_1,s_2}(\mathcal{L})$ is an L-space and therefore by applying again Lemma \ref{lemma intervallo}, since $r_1\geq 1$, we have that 
    $$
    \left[\floor*{r_1},\infty\right]\subset L(S^3_{\bullet,s_2}(\mathcal{L}))
    $$
    and this is exactly what we wanted.
    \item \emph{$r_1<1$ and $r_2\geq 1$:} this case is completely analogous to the previous one.
    \item \emph{$r_1<1$ and $r_2<1$:} also in this case we have that 
    $$
    (0,\infty]=\left(\floor*{r_2},\infty\right]\subset L(S^3_{r_1,\bullet}(\mathcal{L}))=\overline{\Q}\setminus\{0\}.
    $$
    As a consequence, for any $s_2>0$ we have that $S^3_{r_1,s_2}(\mathcal{L})$ is an L-space and therefore by applying again Lemma \ref{lemma intervallo}, since $0<r_1< 1$, we have that 
    $$
    (0,\infty]=\left(\floor*{r_1},\infty\right]\subset L(S^3_{\bullet,s_2}(\mathcal{L}))=\overline{\Q}\setminus\{0\}.
    $$
\end{itemize}
This concludes the proof.
\end{proof}

\begin{figure}[H]
   \centering
    \includegraphics[width=0.6\textwidth]{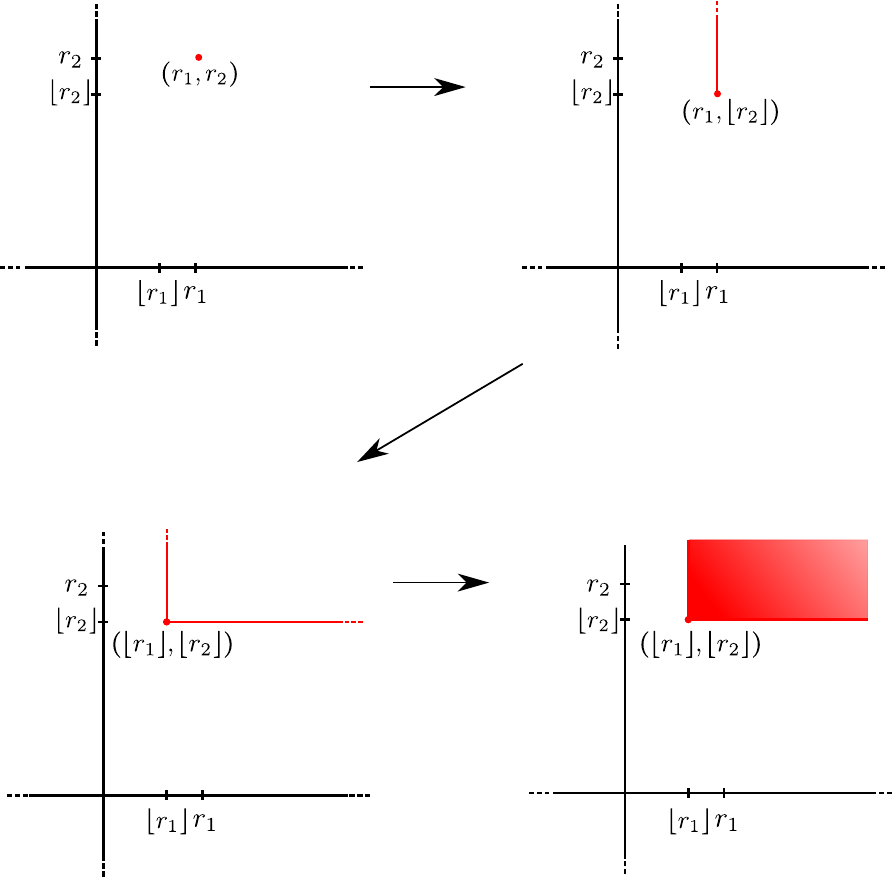}
    \caption{A pictorial sketch of the proof of Lemma \ref{lemma razionale implica L space}.}
    \label{proof}
\end{figure}

\begin{remark}
It follows from Theorem \ref{struttura L space surgery} that if $\mathcal{L}$ has two components, then in the previous lemma the case $0<r_1<1$ or $0<r_2<1$ cannot occur.
\end{remark}

Before proving Theorem \ref{struttura L space surgery} we recall the following theorem from \cite{GLM}.

\begin{theorem}\label{Gorsky}(\cite{GLM})
Assume that $\mathcal{L}$ is a non-trivial L–space link with unknotted components
and linking number zero. Then there exist non-negative integers $b_1,b_2$ such that for $p_1,p_2\in \Z$ we have that $S^3_{p_1,p_2}(\mathcal{L})$ is an L–space if and only if $p_1>2b_1$ and
$p_2>2b_2$.

\end{theorem}

\begin{proof}[Proof of Theorem \ref{struttura L space surgery}]
We know from Lemma \ref{lemma razionale implica L space} that $\mathcal{L}$ is an L-space link. Therefore we can apply Theorem \ref{Gorsky} and deduce that there exist non-negative integers $b_1,b_2$ such that
$$
L(\mathcal{L})\cap\Z^2=[2b_1+1,\infty)\times [2b_2+1,\infty)\cap \Z^2
$$
Exactly as in the proof of Lemma \ref{lemma razionale implica L space}, we can use Lemma \ref{lemma intervallo} to deduce that 
$$
L(\mathcal{L})\supset\big{(}[2b_1+1,\infty]\times[2b_2+1,\infty]\big{)}\cup \big{(}\{\infty\}\times \Q^*\big{)} \cup  \big{(}\Q^*\times\{\infty\}\big{)}.
$$
and therefore we only have to prove that this inclusion is an equality.

Suppose on the contrary that there exists an L-space surgery slope $\left(r_1,r_2\right)$, with $r_1,r_2$ rationals, such that
$$
\left(r_1,r_2\right)\notin\big{(}[2b_1+1,\infty]\times[2b_2+1,\infty]\big{)}\cup \big{(}\{\infty\}\times \Q^*\big{)} \cup  \big{(}\Q^*\times\{\infty\}\big{)}.
$$
We suppose that $r_1<2b_1+1$. The case $r_2< 2b_2 +1$ can be solved in the same way. We have the following cases:
\begin{itemize}
    \item $1\leq r_1<2b_1+1$.
    
    By virtue of Lemma \ref{lemma intervallo} we have that $\left[\floor*{r_1},\infty\right]$ is contained in $L(S^3_{\bullet,r_2})$. This implies that $S^3_{\floor*{r_1},{\bullet}}$ is Floer simple and therefore, by applying again Lemma \ref{lemma intervallo}, we deduce that it admits integral L-space filling slopes. In this way we produce a point in
    $$
    (L(\mathcal{L})\cap\Z^2))\setminus \Big{(}\big{(}[2b_1+1,\infty)\times [2b_2+1,\infty)\big{)}\cap \Z^2\Big{)}
    $$ contradicting Theorem \ref{Gorsky}.
    \item $r_1\in (-1,1)$.
    
      As a consequence of Lemma \ref{lemma intervallo} we have that $L(S^3_{\bullet,r_2})=\overline{\Q}\setminus\{0\}$. Therefore if we fix any negative integer $-m<0$ we have that $r_2\in L(S^3_{-m,\bullet})$ and by applying again Lemma \ref{lemma intervallo} we deduce that there exist integral L-space filling slopes on $S^3_{-k,\bullet}$, contradicting Theorem \ref{Gorsky}.
    \item $r_1\leq-1$.
    
    By applying the same argument used in the first case we have that $L(S^3_{\bullet,r_2})$ contains $\left[\infty,\ceil*{r_1}\right]$. Therefore $S^3_{\ceil*{r_1},{\bullet}}$ admits integral L-space filling slopes, contradicting Theorem \ref{Gorsky}.
    
\end{itemize}
The proof is complete.
\end{proof}

\section{Coorientable Taut Foliations}
In this section we study the existence of taut foliations on the Dehn fillings on the Whitehead link exterior. The main theorem in this section is the following:

\begin{theorem}\label{theorem part 2}
Let $p_1, q_1$ and $p_2,q_2$ be two pairs of non vanishing coprime integers.
Let $S^3_{\sfrac{p_1}{q_1}, \sfrac{p_2}{q_2}}({\rm WL})$ be the $\left(\sfrac{p_1}{q_1}, \sfrac{p_2}{q_2}\right)$-surgery on the Whitehead link. Then
 $S^3_{\sfrac{p_1}{q_1}, \sfrac{p_2}{q_2}}({\rm WL})$ supports a cooriented taut foliation if and only if $\sfrac{p_1}{q_1}< 1$ or $\sfrac{p_2}{q_2}< 1$.
\end{theorem}
In this paper the term \emph{foliation} will refer to codimension-$1$ foliations of class $C^{\infty,0}$, as defined for example in \cite{Candel} and \cite{KR2}. We recall the definition here. We denote with $\text{H}^k$ the $k$-dimensional Euclidean closed half space 
$$
\text{H}^k=\{(x_1,\dots,x_k)\in \R^k\,|\,x_k\geq 0\}.
$$
\begin{definition}
A $\mathbf{C^{\infty,0}}$ \textbf{codimension-1 foliation} $\mathcal{F}$ of a  smooth $3$-manifold $M$ with (possibly empty) boundary is a decomposition of $M$ into the union of disjoint smoothly injectively immersed surfaces, called the \textbf{leaves} of $\mathcal{F}$, together with a collection of charts $(U_i,\phi_i)_{i\in \mathcal{I}}$ covering $M$ such that:
\begin{itemize}
    \item $\phi_i:U_i \rightarrow \mathcal{X}$ is a homeomorphism, where $\mathcal{X}$ is either $\R^2\times \R$ or $\R^2\times \text{H}^1$ or $\text{H}^2\times \R$, with the property that the image of each component of a leaf intersected with $U_i$ is a slice $\R^2\times \{\text{point}\}$ or $\text{H}^2\times \{\text{point}\}$;
    \item all partial derivatives of any order in the variables $x$ and $y$ on the domain of each transition function $\phi_j\phi_i^{-1}$ are continuous; here we have fixed coordinates $(x,y,z)$ on $\mathcal{X}$.
\end{itemize}
\end{definition}
The three local models for a foliation are depicted in Figure \ref{model foliations}, where $\partial \mathcal{X}$ is shaded.

\begin{figure}[H]
    \centering
    \includegraphics[width=0.8\textwidth]{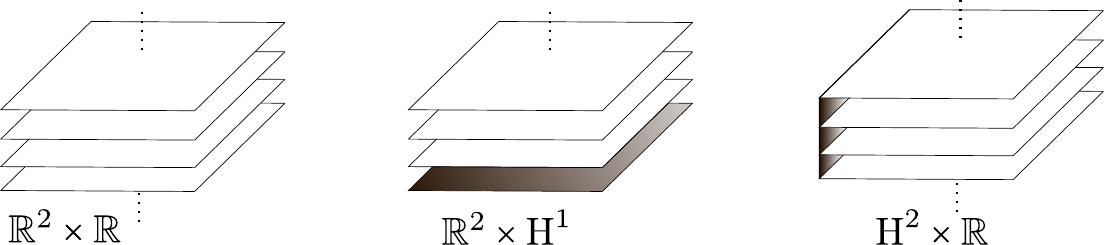}
    \caption{Local models for a foliation.}
    \label{model foliations}
\end{figure}

\begin{remark}
The tangent planes to the leaves of a foliation $\mathcal{F}$ of a $3$-manifold $M$ define a continuous plane subbundle of $TM$, that we denote with $T\mathcal{F}$.
\end{remark}

\begin{definition}
A foliation $\mathcal{F}$ of a $3$-manifold $M$ is \textbf{orientable} if the plane bundle $T\mathcal{F}$ is orientable and is \textbf{coorientable} if the line bundle $TM/T\mathcal{F}$ is orientable.
\end{definition}

\begin{definition}
A foliation $\mathcal{F}$ of a $3$-manifold $M$ is \textbf{taut} if every leaf of $\mathcal{F}$ intersects a closed transversal, i.e. a smooth simple closed curve in $M$ that is transverse to $\mathcal{F}$.
\end{definition}
There are several definitions of tautness and in general they are not equivalent. For details we refer to \cite{CKR}, where also the relations among these different notions are discussed. 
\\

We recall that in order to prove Theorem \ref{theorem part 2} it is enough to prove the ``if'' part, since L-spaces do not support taut foliations (see \cite{OS}, \cite{Bow}, \cite{KR2}), and we have already proved in the previous section that if $\sfrac{p_i}{q_i}\geq 1$ for $i=1,2$, then $S^3_{\sfrac{p_1}{q_1}, \sfrac{p_2}{q_2}}({\rm WL})$ is an L-space.
\\

This section is organised as follows. In Section \ref{Background} we recall the required background regarding branched surfaces and we state the theorem of \cite{L2}. In Section \ref{Constructing taut foliations} we prove Theorem \ref{teorema foliazioni}, regarding the existence of taut foliations on Dehn fillings on manifolds that fiber over the circle with fiber a $k$-holed torus and with some prescribed monodromy. This theorem will be useful to prove that many of the manifolds of Theorem \ref{theorem part 2} support taut foliations. In Section \ref{The Whitehead link case} we conclude the proof of Theorem \ref{theorem part 2}.

\subsection{Background}\label{Background}

In this and in the next sections we will assume familiarity with the basic notions of the theory of train tracks; see \cite{PH} for reference. We only point out that in the cases of our interest, train tracks can also have bigons as complementary regions. 

We now recall some basic facts about branched surfaces. We refer to \cite{FO} and \cite{O} for more details.

\begin{definition}
A \textbf{branched surface with boundary} in a $3$-manifold $M$ is a closed subset $B\subset M$ that is locally diffeomorphic to one of the models in $\R^3$ of Figure \ref{branched surface}a) or to one of the models in the closed half space of Figure \ref{branched surface}b), where $\partial B:= B\cap \partial M$ is represented with a bolded line:
\begin{figure}[H]
    \centering
    \includegraphics[width=0.8\textwidth]{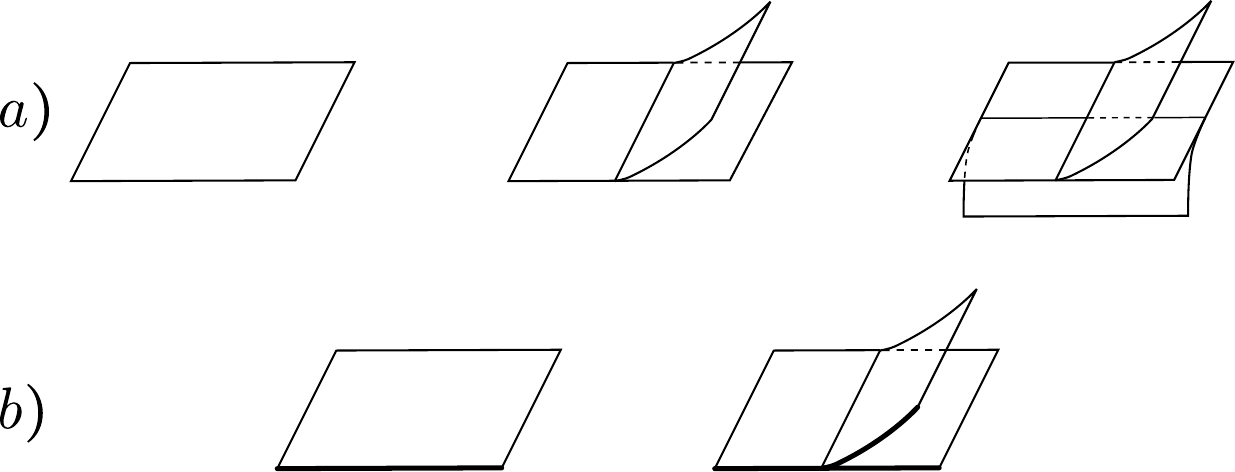}
    \caption{Local models for a branched surface.}
    \label{branched surface}
\end{figure}

\end{definition}

Branched surfaces generalise the concept of train tracks from surfaces to $3$-manifolds and when the boundary of $B$ is non empty it defines a train track $\partial B$ in $\partial M$.

If $B$ is a branched surface it is possible to identify two subsets of $B$: the \textbf{branch locus} and the set of \textbf{triple points} of $B$. The branch locus is defined as the set of points where $B$ is not locally homeomorphic to a surface. It is self-transverse and intersects itself in double points only. The set of triple points of $B$ can be defined as the points where the branch locus is not locally homeomorphic to an arc. For example, the rightmost model of Figure \ref{branched surface}a) contains a triple point.

The complement of the branch locus in $B$ is a union of connected surfaces. The abstract closures of these surfaces under any path metric on $M$ are called the \textbf{branch sectors} of $B$. Analogously the complement of the set of the triple points inside the branch locus is a union of $1$-dimensional connected manifolds. Moreover to each of these manifolds we can associate an arrow in $B$ pointing in the direction of the smoothing, as we did in Figure \ref{cusp directions}. We call these arrows \textbf{branch directions}, or also \textbf{cusp directions}.

\begin{figure}[H]
    \centering
    \includegraphics[width=0.6\textwidth]{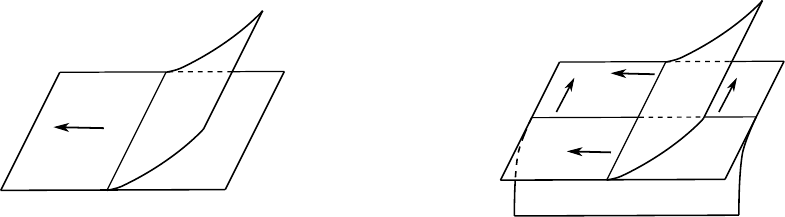}
    \caption{Some examples of cusp directions.}
    \label{cusp directions}
\end{figure}

If $B$ is a branched surface in $M$, we denote with $N_B$ a fibered regular neighbourhood of $B$ constructed as suggested in Figure \ref{regular neighbourhood}.

\begin{figure}[H]
    \centering
    \includegraphics[width=0.6\textwidth]{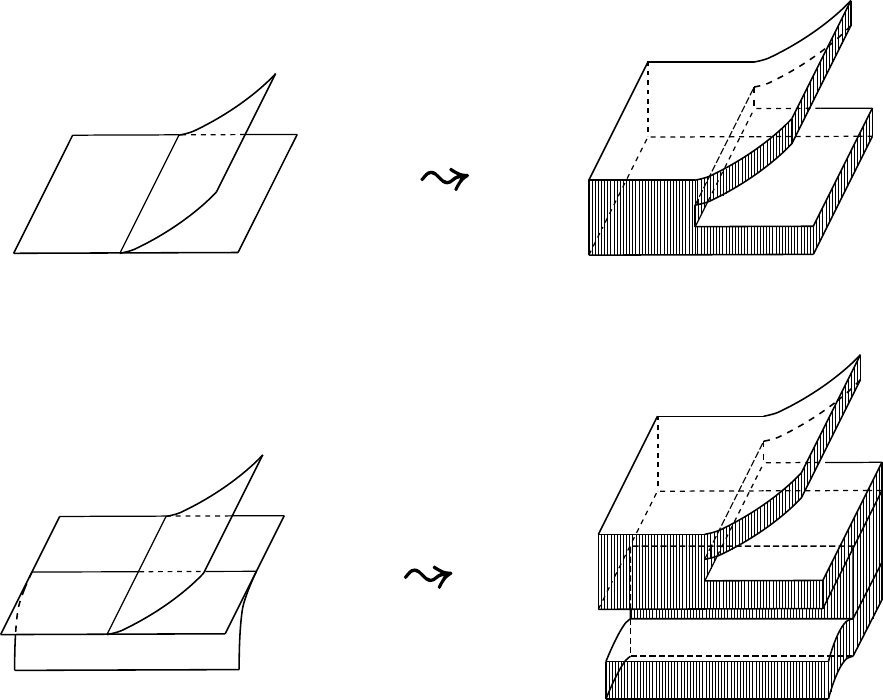}
    \caption{Regular neighbourhood of a branched surface.}
    \label{regular neighbourhood}
\end{figure}
The boundary of $N_B$ decomposes naturally into the union of three compact subsurfaces $\partial_h N_B$, $\partial_v N_B$ and $N_B\cap \partial M$. We call $\partial_h N_B$ the \textbf{horizontal boundary} of $N_B$ and $\partial_v N_B$ the  \textbf{vertical boundary} of $N_B$. The horizontal boundary is transverse to the interval fibers of $N_B$ while the vertical boundary intersects, if at all, the fibers of $N_B$ in one or two proper closed subintervals contained in their interior. If we collapse each interval fiber of $N_B$ to a point, we obtain a branched surface in $M$ that is isotopic to $B$, and the image of $\partial_v N_B$ coincides with the branch locus of such a branched surface.  

We also recall the definition of \emph{splitting}\footnote{This operation is referred to as \emph{restriction} in \cite{O}}.

\begin{definition}
Given two branched surfaces $B_1$ and $B_2$ in $M$ we say that $B_2$ is obtained by \textbf{splitting} $B_1$ if $N_{B_1}$ can be obtained as $N_{B_2}\cup J$, where $J$ is a $[0,1]$-bundle such that $\partial_h J\subset \partial_h N_{B_2}$, $\partial_v J\cap \partial N_{B_2}\subset \partial_v N_{B_2}$ and $\partial J$ meets $\partial N_{B_2}$ so that the fibers agree.
\end{definition}
Figure \ref{splitting} shows two examples of splittings, depicted for the case of $1$-dimensional branched manifolds, i.e. train tracks. 

\begin{figure}[H]
    \centering
    \includegraphics[width=0.6\textwidth]{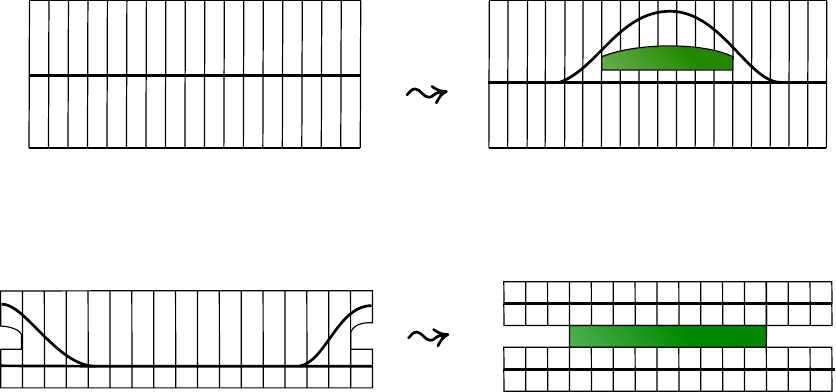}
    \caption{Some examples of splittings. The coloured region is the interval bundle $J$.}
    \label{splitting}
\end{figure}

Branched surfaces provide a useful tool to construct \emph{laminations} on $3$-manifolds. 
\begin{definition}\label{lamination}(see for example \cite{GO})
Let $B$ be a branched surface in a $3$-manifold $M$. A \textbf{lamination carried by B} is a closed subset $\Lambda$ of some regular neighbourhood $N_B$ of $B$ such that $\Lambda$ is a disjoint union of smoothly injectively immersed surfaces, called leaves, that intersect the fibers of $N_B$ transversely. We say that $\Lambda$ is \textbf{fully carried} by $B$ if $\Lambda$ is carried by $B$ and intersects \emph{every} fiber of $N_B$.
\end{definition}

\begin{remark}
Analogously as in Definition \ref{lamination}, if $S$ is a closed oriented surface and $\tau$ is a train track in $S$ we can define what is a lamination (fully) carried by $\tau$. In this case we say that an oriented simple closed curve $\gamma$ is \textbf{realised} by $\tau$ if $\tau$ fully carries a union of finitely many disjoint curves that are parallel to $\gamma$ inside $S$.
\end{remark}

In \cite{L}, Li introduces the notion of \emph{sink disc}.

\begin{definition}
Let $B$ be a branched surface in $M$ and let $S$ be a branch sector in $B$.
We say that $S$ is a \textbf{sink disc} if $S$ is a disc, $S\cap \partial M=\emptyset$ and the branch direction of any smooth curve or arc in its boundary points into $S$.
We say that $S$ is a \textbf{half sink disc} if $S$ is a disc, $S  \cap \partial M\neq \emptyset$ and the branch direction of any smooth arc in $\partial S\setminus \partial M$ points into $S$.
\end{definition}

In Figure \ref{discs}  some examples of sink discs and half sink discs are depicted. The bolded lines represent the intersection of the branched surface with $\partial M$. Notice that if $S$ is a half sink disc the intersection $\partial S \cap \partial M$ can also be disconnected.

\begin{figure}[H]
    \centering
    \includegraphics[width=0.5\textwidth]{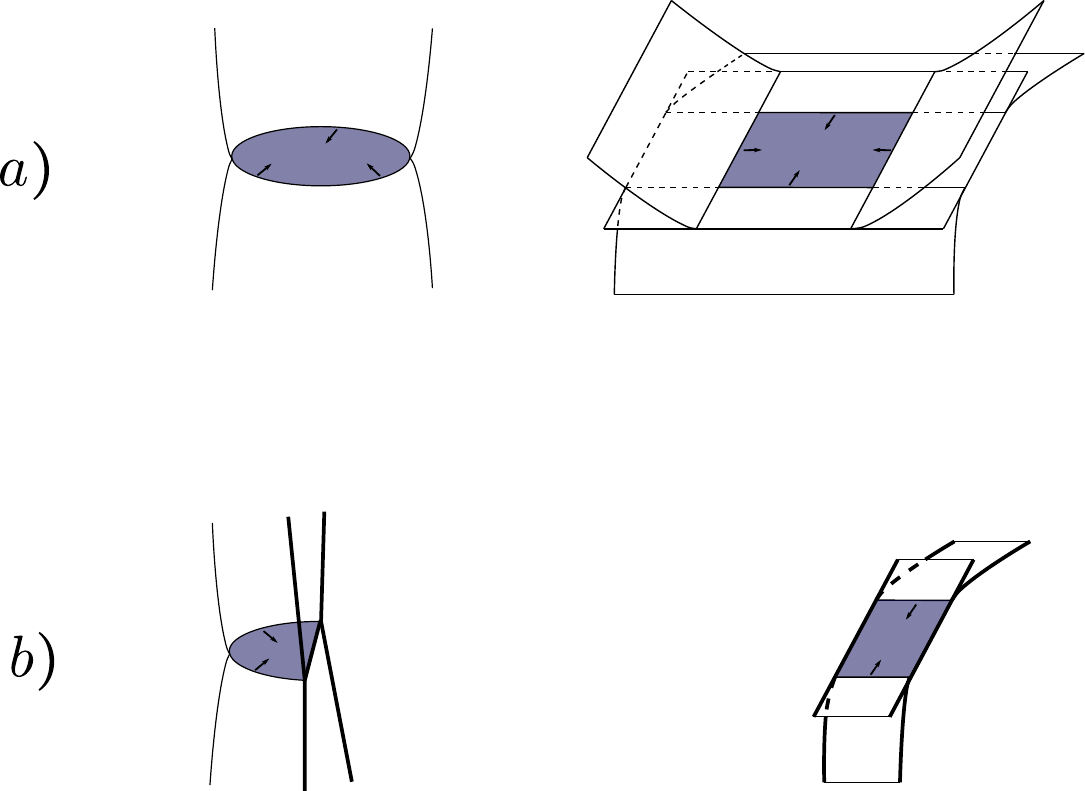}
    \caption{Examples of $a)$ sink discs and $b)$ half sink discs.}
    \label{discs}
\end{figure}

If $B$ contains a sink disc or a half sink disc there is a very simple way to eliminate it, namely it is enough to blow an air bubble in its interior, as in the following figure.

\begin{figure}[H]
    \centering
    \includegraphics[width=0.6\textwidth]{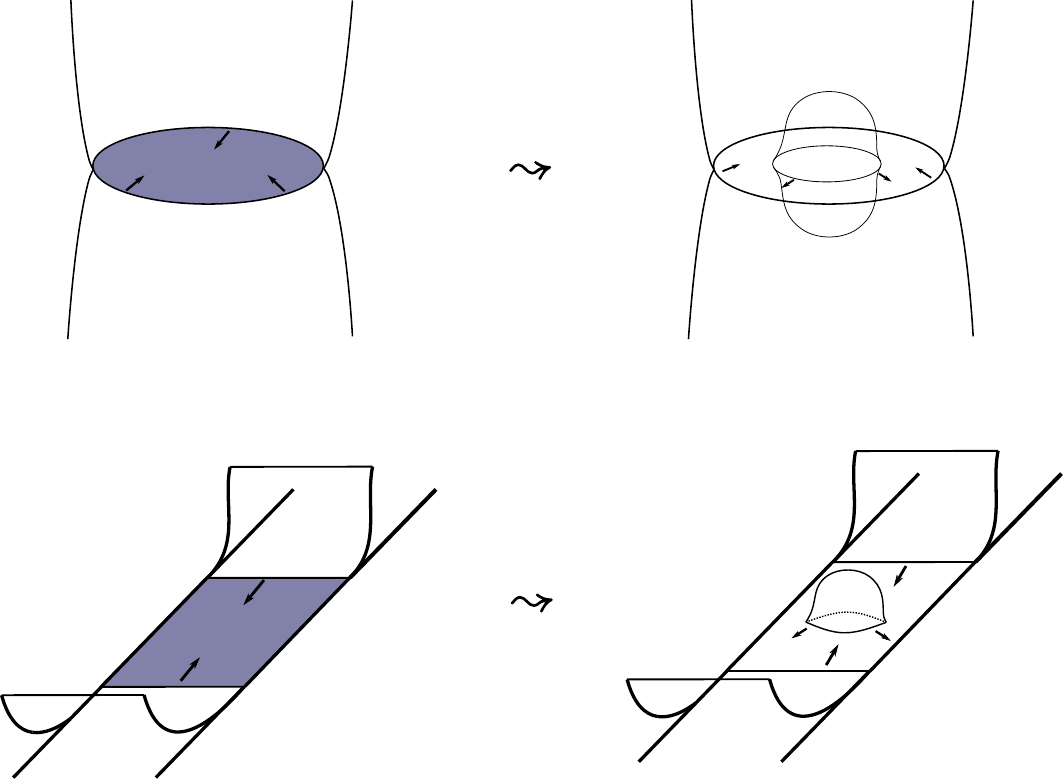}
    \caption{How to eliminate a sink disc or a half sink disc by blowing an air bubble.}
    \label{air bubble}
\end{figure}

We want to avoid this situation. We say that a connected component of $M\setminus {\rm int}(N_B)$ is a $\mathbf{\D^2\times [0,1]}$ \textbf{region} if it is homeomorphic to a ball and its boundary can be subdivided into an annular region, corresponding to a component of $\partial_v N_B$, and two $\D^2$ regions corresponding to components of $\partial_h N_B$. We say that a $\D^2 \times [0,1]$ region is \textbf{trivial} if the map collapsing the fibers of $N_B$ is injective on ${\rm int}(\D^2)\times \{0,1\}$. In this case the image of $\D^2\times \{0,1\}$ via the collapsing map is called a \textbf{trivial bubble} in $B$. Trivial bubbles and trivial $\D^2\times [0,1]$ regions are created when we eliminate sink discs as in Figure \ref{air bubble}.

When $M$ and $B$ have boundary these definitions generalise straightforwardly to the relative case, see \cite{L2}.

In \cite{L}, Li introduces the definition of laminar branched surface and shows that laminar branched surfaces fully carry essential laminations\footnote{For the definition of essential lamination see \cite{GO}, but we will not need their properties for our purposes.}.
In \cite{L2} he generalises this definition to branched surfaces with boundary as follows:

\begin{definition}[\cite{L}]\label{laminar}
Let $B$ be a branched surface in a $3$-manifold $M$. We say that $B$ is \textbf{laminar} if $B$ has no trivial bubbles and the following hold:

\begin{enumerate}
    \item $\partial_h N_B$ is incompressible and $\partial$-incompressible in $M\setminus {\rm int}(N_B)$, and no component of $\partial_h N_B$ is a sphere or a properly embedded disc in $M$;
    \item there is no monogon in $M\setminus {\rm int}(N_B)$, i.e. no disc $D\subset M\setminus {\rm int}(N_B)$ such that $\partial D=D\cap N_B=\alpha\cup \beta$, where $\alpha$ is in an interval fiber of $\partial_v N_B$ and $\beta$ is an arc in $\partial_h N_B$;
    \item $M\setminus {\rm int}(N_B)$ is irreducible and $\partial M\setminus {\rm int}(N_B)$ is incompressible in $M\setminus {\rm int}(N_B)$;
    \item $B$ contains no Reeb branched surfaces (see \cite{GO} for the definition);
    \item $B$ has no sink discs or half sink discs.
\end{enumerate}
\end{definition}
Since $\partial_h N_B$ is not properly embedded in $M\setminus {\rm int}(N_B)$ we explain more precisely the request of $\partial$-incompressibility in $1.$ : we require that if $D$ is a disc in $M\setminus {\rm int}(N_B)$ with ${\rm int}(D)\subset M\setminus N_B$ and $\partial D=\alpha\cup \beta$ where $\alpha$ is an arc in $\partial_h N_B$ and $\beta$ is an arc in $\partial M$, then there is a disc $D'\subset \partial_h N_B$ with $\partial D'=\alpha\cup \beta'$ where $\beta'=\partial D'\cap \partial M$.

The following theorem of \cite{L2} will be used profusely in this section.

\begin{theorem}\cite{L2}\label{boundary train tracks}
Let $M$ be an irreducible and orientable $3$-manifold whose boundary is union of $k$ incompressible tori $T_1,\cdots, T_k$. Suppose that $B$ is a laminar branched surface in $M$ such that $\partial M\setminus \partial B$ is a union of bigons. Then for any multislope $(s_1,\dots, s_k)\in \overline{\Q}^k$ that is realised by the train track $\partial B$, if $B$ does not carry a torus that bounds a solid torus in $M(s_1,\dots,s_k)$, there exists an essential lamination $\Lambda$ in $M$ fully carried by $B$ that intersects $\partial M$ in parallel simple curves of multislope $(s_1,\dots, s_k)$. Moreover this lamination extends to an essential lamination of the filled manifold $M(s_1,\dots, s_k)$.
\end{theorem}

\begin{remark}
The statement of the Theorem \ref{boundary train tracks} is slightly more detailed than the version of \cite{L2}. The details we have added come from the proof of Theorem \ref{boundary train tracks}. In fact the idea of the proof is to split the branched surface $B$ in a neighbourhood of $\partial M$ so that it intersects $T_i$ in parallel simple closed curves of slopes $s_i$, for $i=1,\dots k$. In this way, when gluing the solid tori, we can glue  meridional discs of these tori to $B$ so to obtain a branched surface $B(s_1,\dots, s_k)$ in $M(s_1,\dots, s_k)$ that is laminar and that therefore by virtue of \cite[Theorem~1]{L} fully carries an essential lamination. In particular, this essential lamination is obtained by gluing the meridional discs of the solid tori to an essential lamination in $M$ that intersects $T_i$ in parallel simple closed curves of slopes $s_i$, for $i=1,\dots, k$.
\end{remark}

\begin{remark}
In \cite{L2} the statement of the theorem is given for $M$ with connected boundary, but as already noticed in \cite{KR1} if $M$ has multiple boundary components we can split $B$ in a neighbourhood of each boundary tori $T_i$ and the same proof of \cite{L2} works.
\end{remark}

\subsection{Constructing taut foliations}\label{Constructing taut foliations}

The aim of this subsection is to prove Theorem \ref{teorema foliazioni}, that concerns the existence of taut foliations on Dehn fillings on manifolds that fiber over the circle with fiber a $k$-holed torus and with some prescribed monodromy. To do this we will recall a very simple, yet useful, way to build branched surfaces.
First of all we fix some notations and recall the definition of fibered link.
\\

Given an oriented surface $S$ with (possibly empty) boundary and $h:S\rightarrow S$ an orientation preserving homeomorphism of $S$ fixing $\partial S$ pointwise we denote with $M_h$ the mapping torus of $h$
$$
M_h=\frac{S\times [0,1]}{(h(x),0)\sim (x,1)}.
$$
We orient $S\times [0,1]$ as a product and we orient $M_h$ with the orientation induced by $S\times [0,1]$. We also identify $S$ with its image in $M_h$ via the map
\begin{align*}
S&\rightarrow S\times \{0\}\subset M_h\\
x&\mapsto (x,0). 
\end{align*}
The homeomorphism $h$ is called the \textbf{monodromy} of $M_h$.

\begin{definition}
Let $L$ be an oriented link in $S^3$. We say that $L$ is \textbf{fibered} if there exists a Seifert surface $S$ for $L$, an orientation preserving homeomorphism $h$ of $S$ fixing $\partial S$ pointwise and an orientation preserving homeomorphism
$$
\chi: S^3\setminus {\rm int}(N_L) \rightarrow M_h,
$$
where $N_L$ denotes a tubular neighbourhood of $L$ in $S^3$, so that 
\begin{itemize}
    \item $\chi_{|S}$ is the inclusion $S\subset M_h$;
    \item $\chi(m_i)=\{x_i\}\times [0,1]$, where $m_i$ is a meridian for the $i$-th component of $L$ and $x_i\in \partial S$ is a point.
\end{itemize}
\end{definition}

We want to apply Theorem \ref{boundary train tracks} to construct laminations on $S^3_{\sfrac{p_1}{q_1}, \sfrac{p_2}{q_2}}({\rm WL})$ when  $\sfrac{p_i}{q_i}< 1$ for at least one $i\in \{1,2\}$ and then promote them to taut foliations. To do this we will define some branched surfaces in the exterior of the Whitehead link and then study their boundary train tracks.

The construction of these branched surfaces relies on the fact that the Whitehead link is a fibered link.
In fact, as Figure \ref{hopf plumbing} shows, the Whitehead link can be obtained as the boundary of a surface $F$ that is a torus with two open discs removed. This torus is obtained by a sequence of three Hopf plumbings and this implies by standard results (see \cite{G1,S1}) that $F$ is a fiber surface for ${\rm WL}$ with monodromy $h$ given by $h=\tau_0\tau_1\tau_2^{-1}$, where $\tau_i$ denotes the positive Dehn twist along the curve $\gamma_i$ and where the factorisation of $h$ should be read from right to left.

\begin{figure}[H]
    \centering
    \includegraphics[width=0.7\textwidth]{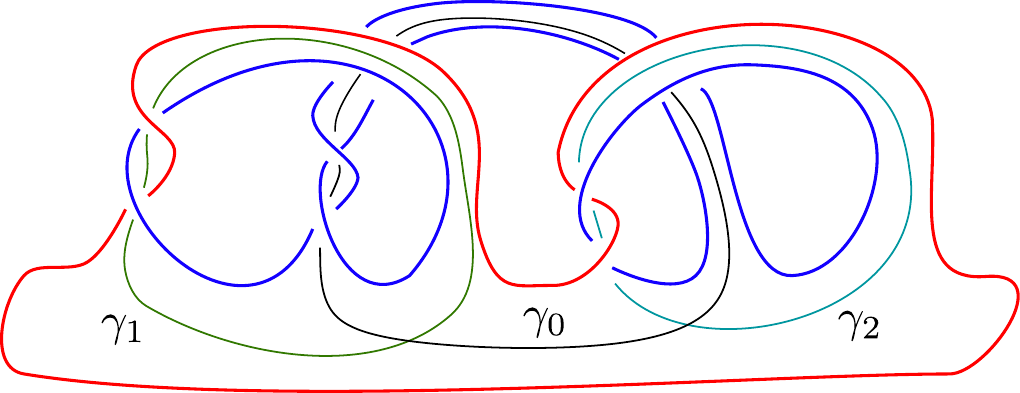}
    \caption{ The Whitehead link is the union of the red curve and the blue curve. The curves $\gamma_0,\gamma_1,\gamma_2$ lie on the Seifert surface that should be evident from the picture.}
    \label{hopf plumbing}
\end{figure}

Before focusing on the specific example of the Whitehead link exterior, we work in a more general setting. 

Let $S$ be an oriented surface with boundary and let $h$ be an orientation preserving homeomorphism of $S$ fixing $\partial S$ pointwise.
We consider some pairwise disjoint properly embedded arcs $\alpha_1,\dots, \alpha_k$ in $S$ and the discs $\overline{D}_i=\alpha_{i}\times[0,1]\subset S\times [0,1]$. Each of these discs has a ``bottom'' boundary, $\alpha_i\times \{0\}$, and a ``top'' boundary, $\alpha_i\times \{1\}$. When we consider the images of these discs in $M_h$ under the projection map
$$
S\times [0,1]\rightarrow M_h
$$
we have that the bottom and top boundaries become respectively $\cup_{i}\alpha_i\subset S$ and $\cup_{i}h(\alpha_i)\subset~S$.

We can isotope simultaneously the discs $\overline{D}_i$'s in a neighbourhood of $S\times \{1\}\subset S\times [0,1]$ so that when projected to $M_h$ their top boundaries define a family of arcs $\{\smash{\widetilde{h(\alpha_i)}}\}_{i=1,\dots k}$ in $S$ such that for each $i,j\in \{1,\dots, k\}$ the intersection between $\alpha_i$ and $\smash{\widetilde{h(\alpha_j)}}$ is transverse and minimal. Notice that each arc $\smash{\widetilde{h(\alpha_i)}}$ is isotopic as a properly embedded arc to $h(\alpha_i)$. We also denote with $D_i$ the projected perturbed disc contained in $M_h$.
\newline

If we assign (co)orientations to these discs, since $S$ is (co)oriented, we can smoothen $S\cup D_1 \cup \cdots \cup D_k$ to a branched surface $B$ by imposing that the smoothing preserves the coorientation of $S$ and of the discs. In particular, each disc has two possible coorientations and therefore it can be smoothed in two differents ways. This operation is demonstrated in Figure \ref{smoothings}, where $S$ is a torus with an open disc removed.

\begin{figure}[H]
    \centering
    \includegraphics[width=0.9\textwidth]{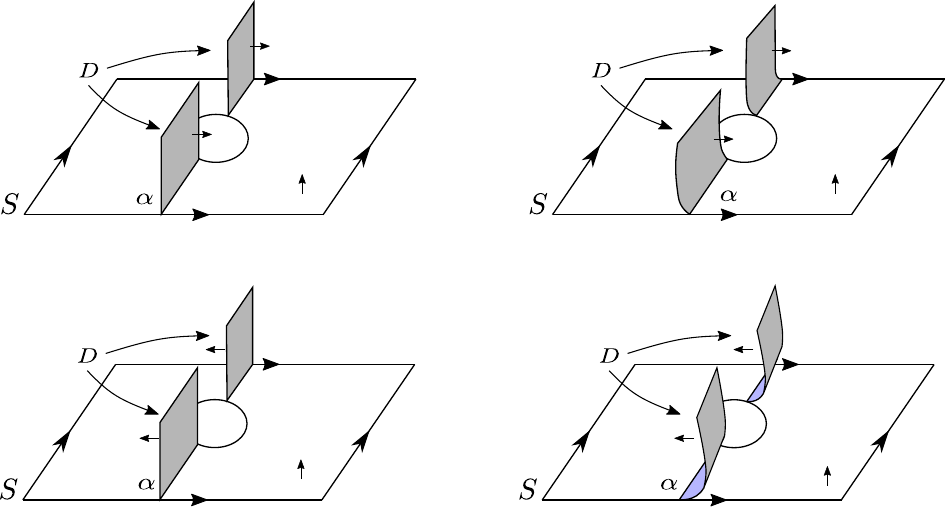}
    \caption{How to smoothen $S\cup D$ according to the coorientations.}
    \label{smoothings}
\end{figure}

We prove the following lemma, that is only implicit in \cite{KR1}.
\begin{lemma}\label{lemma laminar}
Let $S$ be a connected and oriented surface with boundary and let $h$ be an orientation preserving homeomorphism of $S$ fixing $\partial S$ pointwise. Let $\{\alpha_i\}_{i=1,\dots, k}\subset S$ be pairwise disjoint properly embedded arcs in $S$ and suppose that $S\setminus \cup_{i=1}^k{\alpha_i}$ has no disc components. Denote with $D_i$'s the discs in $M_h$ associated to the arcs $\alpha_i$'s in the way described above and fix a coorientation for these discs. Let $B=S\cup D_1\cup \cdots \cup D_k$ denote the branched surface in $M_h$ obtained by smoothing according to these coorientations. Then $B$ has no trivial bubbles and satisfies conditions $1,2$,$3$ and $4$ of Definition \ref{laminar}.

\end{lemma}

\begin{proof}
We denote the mapping torus $M_h$ with $M$. We fix for each arc $\alpha_i$ a tubular neighbourhood $N_{\alpha_i}$ in $S$ and we denote with $S'$ the surface $S\setminus \cup_{i=1}^k {\rm int}(N_{\alpha_i})$. The first observation is that by construction we have
$$
M\setminus {\rm int}(N_B) \cong S'\times [0,1]
$$
with a homeomorphism that identifies
$$
\partial_h N_B=S'\times \{0,1\}.
$$
and 
$$
\partial_v N_B=\partial'S'\times [0,1]
$$
where $\partial'S'$ denotes the closure of $\partial S'\setminus \partial M$.

Basically, the proof follows from the fact that  $M\setminus  {\rm int}(N_B)$ is homeomorphic to $S'\times [0,1]$ and that $S'$ has no discs components.

First of all, we notice that since by hypothesis $S\setminus \cup_{i=1}^k{\alpha_i}$ has no discs components, there are no $\D^2\times [0,1]$ regions in $M\setminus  {\rm int}(N_B)$ and in particular no trivial bubbles.
We now verify that conditions $1-4$ of Definition \ref{laminar} hold.

\begin{enumerate}
    \item \begin{itemize}
        \item \emph{The horizontal boundary $\partial_h N_B$ is incompressible in $M\setminus {\rm int}(N_B)$:} this follows from the fact that the inclusions of $S'\times \{0\}$ and $S'\times \{1\}$ in $M\setminus {\rm int}(N_B)$ are homotopy equivalences. In particular, if a simple closed curve in $\partial_h N_B$ bounds a disc in $M\setminus {\rm int}(N_B)$ then it must be nullhomotopic in $\partial_h N_B$ and nullhomotopic simple closed curves in surfaces always bound embedded discs.
        \item \emph{The horizontal boundary $\partial_h N_B$ is $\partial$-incompressible in $M\setminus {\rm int}(N_B)$:}
        suppose that there is a disc $\Delta\subset M\setminus {\rm int}(N_B)$ such that ${\rm int}(\Delta)\subset M\setminus N_B$ and $\partial \Delta=a\cup b$, where $a$ is an arc in $\partial_h N_B$ and $b=\partial \Delta\cap \partial M$. 
        We have to find a disc $\Delta'\subset \partial_h N_B$ with $\partial \Delta'=a\cup b'$ where $b'=\partial \Delta'\cap \partial M$.
        
        Without loss of generality we can suppose that $a\subset S'\times \{0\}$. The arc $b$ is an arc in $\partial M\setminus {\rm int}(N_B)$ with both endpoints in $S'\times \{0\}$ and since the connected components of $\partial M\setminus {\rm int}(N_B)$ are either discs or annuli, there exists a homotopy in $\partial M\setminus {\rm int}(N_B)$, relative to the boundary, from the arc $b$ to an arc $b'\subset (S'\times \{0\})\cap \partial M$. In particular since the simple closed curve $a\cup b$ is nullhomotopic in $M\setminus {\rm int}(N_B)$, the curve $a\cup b'$ is nullhomotopic as well.
        \newline
        To conclude it is enough to observe that since the inclusion of $S'\times \{0\}$ in $M\setminus {\rm int}(N_B)$ is a homotopy equivalence, the simple closed curve $a\cup b'$ bounds a disc $\Delta'$ in $S'\times \{0\}$.

    \item \emph{No component of the horizontal boundary is a sphere or a properly embedded disc:} this follows by our hypotheses.
    \end{itemize}

    \item \emph{there is no monogon in $M\setminus {\rm int}(N_B)$:}
    this is a consequence of the fact that the branched surface $B$ admits a coorientation.
    
    \item 
    \begin{itemize}
        \item \emph{$M\setminus {\rm int} (N_B)$ is irreducible:} this is a consequence of the fact that each component of $M\setminus {\rm int} (N_B)$ is the product of a surface with boundary with $[0,1]$.
        \item \emph{$\partial M\setminus {\rm int}(N_B)$ is incompressible in $M\setminus {\rm int}(N_B)$:} consider any boundary component $T$ of $M$. By construction $T\setminus {\rm int}(N_B)$ is a union of discs or an annulus (in case there are no endpoints of the arcs $\alpha_i$ on $T$). In the former case, $T\setminus {\rm int}(N_B)$ is obviously incompressible in $M\setminus N_B$, while in the latter it is compressible if and only if it is the boundary of $S'\times [0,1]$ and $S'\times [0,1]$ is diffeomorphic to $\D^2\times [0,1]$, but this would contradict our hypotheses.
    \end{itemize} 
    
    \item \emph{$B$ contains no Reeb branched surfaces:} the presence of a Reeb branched surface would imply that some of the complementary regions of ${\rm int}(N_B)$ are $\D^2\times [0,1]$ regions (see \cite{GO}) and we have already observed that there are no such regions.
\end{enumerate}
The proof is complete.
\end{proof}

All the branched surfaces we will use are obtained with the previous construction. One problem that one has to face is that there could be sink discs in such branched surfaces. In \cite{KR1} the authors present a useful procedure to build splittings of branched surfaces constructed in this way that are without sink discs, and part of the results that we are going to obtain can also be proved with the methods there presented (see Remark \ref{remark kalelkar roberts}). However to be able to construct taut foliations on all the manifolds of the statement of Theorem \ref{theorem part 2} we will need to find a way to build our branched surfaces that is slightly different from the one presented in \cite{KR1}.
\\

We fix some notation. We suppose that $S$ is a torus with $k$ open discs removed. We consider the curves $\gamma_0, \gamma_1, \dots, \gamma_{k}$ and we label the boundary components of $S$ with numbers in $\{1,\dots, k\}$ as in Figure \ref{torus_k}. We also orient $S$ so that the orientation induced on the boundary components is the one of the figure.

\begin{figure}[H]
   \centering
    \includegraphics[width=0.45\textwidth]{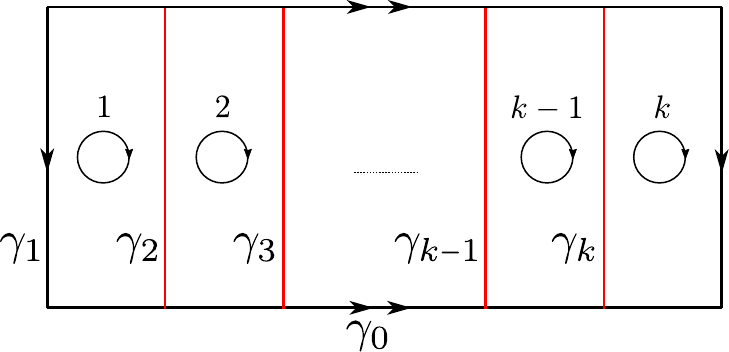}
    \caption{The oriented torus $S$ with the labelled boundary components.}
    \label{torus_k}
\end{figure}

We denote with $\tau_i$ the positive Dehn twist along the curve $\gamma_i$. Notice that since $\gamma_i\cap\gamma_j=\emptyset$ for $i\ne j$ and $i,j \in\{1,\dots, k\}$ we have that $\tau_i\tau_j=\tau_j\tau_i$ for $i,j=1,\dots, k$.

We focus on homeomorphisms of $S$ of the following type: 
$$
h=\tau_0^{a_0}\tau_1^{a_1}\cdots\tau_k^{a_k}\quad a_0\in \Z, \, a_i\in \Z\setminus{\{0\}} \text{ for } i\ne 0
$$
where the factorisation of $h$ should be read from right to left.

We fix the following convention:
\newline

\textbf{Convention:}
\emph{the indices $1,\dots, k$ have to be considered ordered cyclically; so we set $a_{k+1}=a_1$ and think of $a_1$ as consecutive to $a_{k}$.}
\newline

Let $\partial_i S$ denote the boundary component of $S$ labelled with $i$. Given such a homeomorphism $h$ we assign to $\partial_i S$ a label with the following rule:

\begin{itemize}
    \item we assign to $\partial_i S$ the label $p_{+}$ if $a_i$ and $a_{i+1}$ are both positive;
    \item we assign to $\partial_i S$ the label $p_-$ if $a_i$ and $a_{i+1}$ are both negative;
    \item we assign to $\partial_i S$ the label $n$ if $a_i$ and $a_{i+1}$ have different signs.
\end{itemize}
Figure \ref{label} shows an example. Notice that in this example, following our convention, to assign a label to $\partial_3 S$ we have to check the signs of $a_3$ and $a_1$, since $a_1$ is consecutive to $a_3$. Also notice that when $k=1$,  \emph{i.e.} $S$ has only one boundary component we have that $\partial_1 S$ has label $p_+$ when $a_1$ is positive and label $p_-$ when $a_1$ is negative.

\begin{figure}[H]
   \centering
    \includegraphics[width=0.45\textwidth]{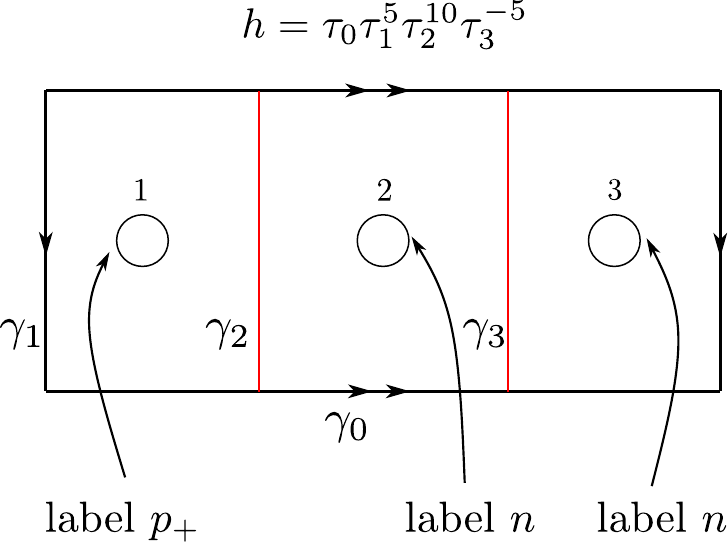}
    \caption{Example with $h=\tau_0\tau_1^5\tau_2^{10}\tau_3^{-5}$.}
    \label{label}
\end{figure}

Finally we assign to each boundary component $\partial_i S$ two intervals $I_i$ and $J_i$ in $\overline{\Q}$ in the following way:

\begin{itemize}
    \item if $\partial_i S$ has label $p_+$ we set $I_i=J_i=(\infty, 1)$;
    \item if $\partial_i S$ has label $p_-$ we set $I_i=J_i=(-1, \infty)$;
    \item let $i_1<\dots< i_{2c}$ the indices of the boundary components labelled with $n$. We set $I_{i_a}=(\infty, 0)$ when $a\in \{1,\dots, 2c\}$ is odd and we set $I_{i_a}=(0,\infty)$ when  $a\in \{1,\dots, 2c\}$ is even.
    
    Therefore we have $I_{i_1}=(\infty, 0)$, $I_{i_2}=(0,\infty)$, $I_{i_3}=(\infty, 0)$ and so on.
    
    On the contrary, we set $J_{i_a}=(0,\infty)$ when $a\in \{1,\dots, 2c\}$ is odd and $J_{i_a}=(\infty,0)$ when $a\in \{1,\dots, 2c\}$ is even.
\end{itemize}

\begin{ex}
In the example of Figure \ref{label} we have 
$$
I_1=J_1=(\infty,1)
$$
$$
I_2=(\infty,0)\quad J_2=(0,\infty)
$$
$$
I_3=(0,\infty)\quad J_3=(\infty,0).
$$

\end{ex}

We are now ready to state the theorem. In the statement of the theorem, for each boundary torus $T_i$ of the manifold $M_h$ we have fixed as longitude the oriented curve $\partial_i S$ and as meridian the image in $M_h$ of the curve $\{x_i\}\times[0,1]$, oriented as $[0,1]$, where $x_i\in \partial_i S$.

\begin{theorem}\label{teorema foliazioni}
Let $S$ be a $k$-holed torus as in Figure \ref{torus_k} and let $h$ be a homeomorphism of $S$ of the following form:
$$
h=\tau_0^{a_0}\tau_1^{a_1}\cdots\tau_{k}^{a_{k}}
$$
where $a_0\in \Z$ and $a_i\in \Z\setminus{\{0\}}$
for $i=1,\dots, k$.
Then: 
\begin{enumerate}
\item if $a_0>0$ (resp. $a_0<0$) then $M_h(s_1,\dots, s_k)$ supports a coorientable taut foliation for each multislope $(s_1,\dots, s_k)\in (\infty, 1)^k$ (resp. $(-1, \infty)^k$);
\item 
for any multislope $(s_1,\dots, s_k)\in (I_1\times\cdots \times I_k)\cup (J_1\times\cdots \times J_k)$, the filled manifold $M_h(s_1,\dots, s_k)$ supports a coorientable taut foliation, where the intervals $I_i$'s and $J_i$'s are the ones described above.

\end{enumerate}
\end{theorem}

\begin{remark}
Notice that if $h'$ is conjugated in ${\rm MCG}(S,\partial S)$ to a homeomorphism $h$ that satisfies the hypotheses of Theorem \ref{teorema foliazioni}, then the conclusion of the theorem holds also for $M_{h'}$.
\end{remark}
\begin{remark}
Notice that the first part of the theorem does not cover the case $a_0=0$. We did not investigate further this case, but for our purpose the statement of Theorem \ref{teorema foliazioni} will be sufficient.
\end{remark}

\subsubsection{Proof of the first part of Theorem \ref{teorema foliazioni}}

To prove Theorem \ref{teorema foliazioni} we will build branched surfaces in $M_h$ satisfying the hypotheses of Theorem \ref{boundary train tracks} by following the construction presented before Lemma \ref{lemma laminar}.
\\

We start by proving the first part of the theorem. 
We define a branched surface as follows. We consider the parallel arcs $\alpha_1,\dots, \alpha_k$ depicted in Figure \ref{parallel arcs}.

\begin{figure}[H]
   \centering
    \includegraphics[width=0.45\textwidth]{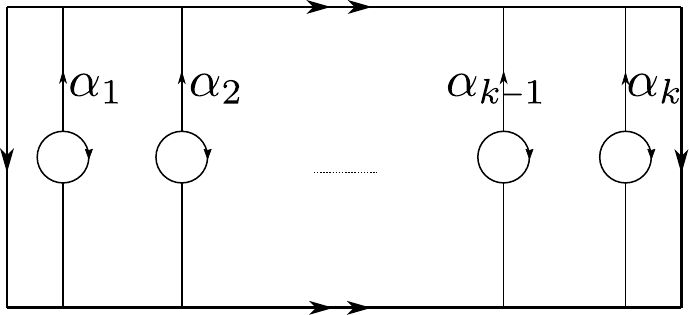}
    \caption{The parallel arcs $\alpha_1\dots, \alpha_k$}
    \label{parallel arcs}
\end{figure}

We consider the discs $\alpha_i\times [0,1]$, perturb them in a neighbourhood of $S\times \{1\}$ as explained in the discussion before Lemma \ref{lemma laminar}, and project them to the mapping torus $M_h$. We consider the (co)oriented branched surface $B$ in $M_h$ obtained by adding these discs $D_i$ to the surface $S$. The discs $D_i$'s are oriented so that the orientation on their boundary induces the given orientation on the arcs $\alpha_i$'s. For an example, see Figure \ref{no sink}, where also some cusp directions are showed. A good way to deduce the cusp directions along the arcs $\alpha_i$'s and $h(\alpha_i)$'s is the following: they point to the right along the arcs $\alpha_i$'s and they point to the left along the arcs $h(\alpha_i)$'s, where the latter are oriented as the image of the arcs $\alpha_i$'s.
\begin{figure}[H]
   \centering
    \includegraphics[width=0.55\textwidth]{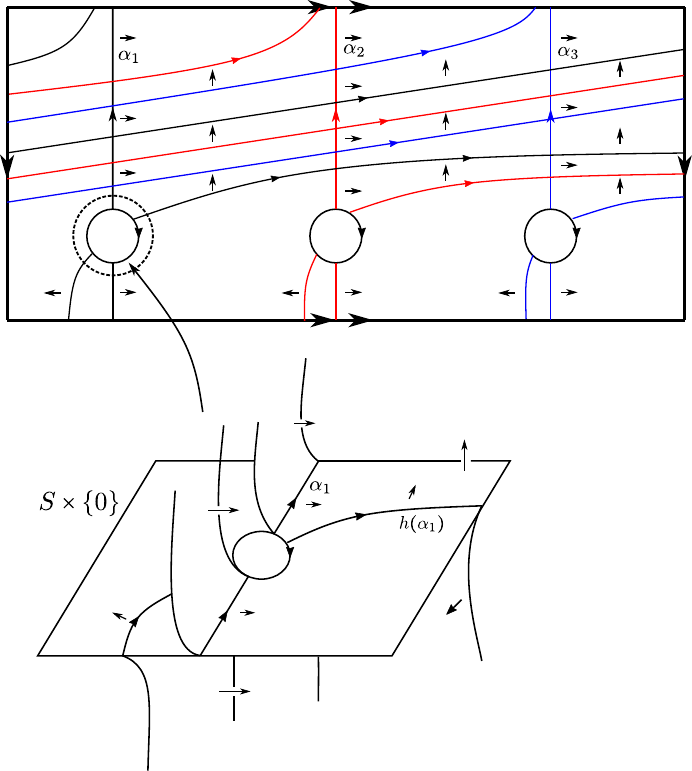}
    \caption{In this example, $a_0=2$ and $k=3$. We also show the details of our branched surface in a neighbourhood of $\partial_1 S$. }
    \label{no sink}
\end{figure}

\begin{lemma}\label{B is laminar}
The branched surface $B$ is laminar and satisfies the hypotheses of Theorem \ref{boundary train tracks}. 
\end{lemma}
\begin{proof}
Thanks to Lemma \ref{lemma laminar} to prove that $B$ is laminar it is enough to prove that $B$ contains no sink discs or half sink discs.
We prove that:

\begin{itemize}
    \item \emph{B contains no sink discs or half sink discs:} there are $k$ sectors of $B$ that are half discs and that coincide with the discs $D_i$'s; these sectors are never sink by construction (see Figure \ref{smoothings}).
    
    The other sectors coincide with the abstract closures of the connected components of $S\setminus{\left(\bigcup_i \alpha_i \cup \smash{\widetilde{h(\alpha_i)}}\right)}$. Being $a_0$ non-zero, these sectors are discs and half discs\footnote{When the product $k|a_0|$ satisfies $k|a_0|\leq 3$ there are only half disc sectors.}.
    We can organise these sectors in the following way. We refer to Figure \ref{no sink} to visualise the situation. If we cut $S$ along the arcs $\alpha_i$'s we obtain $k$ oriented annuli $A_1,\dots, A_k$, so that $\partial A_i\supset-\alpha_i \cup \alpha_{i+1}$ for $i\in \{1, \dots, k\}$, where $-\alpha_i$ denotes the arc $\alpha_i$ with the opposite orientation. Also notice that the cusp directions along $\alpha_i$ point inside $A_i$ and the cusp directions along $\alpha_{i+1}$ point outside $A_i$.
    
    It follows by the definition of the arcs $\alpha_i$ that $h(\alpha_i)=\tau_0^{a_0}(\alpha_i)$.
    Each of these annuli intersects $h(\alpha_j)$ in $|a_0|$ subarcs, for each $j=1,\dots, k$. Therefore, when we cut along the $h(\alpha_i)$'s we subdivide each of this annuli in $k|a_0|$ discs and these discs coincide with the sectors of $B$ in $S$. By construction each of these discs is contained in an annulus, say $A_i$, and intersects both $\alpha_i$ and $\alpha_{i+1}$ and therefore there is a cusp direction pointing outside it.

    \item \emph{The only connected compact surface properly embedded in $M_h$ carried by $B$ is $S$:} if $\Sigma$ is a compact surface properly embedded in $M_h$ carried by $B$, then $\Sigma$ induces an integral weight system on $B$; that is to say, $\Sigma$ defines a way to assign to each branch sector of $B$ a non-negative integer such that along each connected component of the branch locus minus the set of triple points of $B$ the weights sum according to the cusp directions, as represented in the following figure.
    
    \begin{figure}[H]
   \centering
    \includegraphics[width=0.25\textwidth]{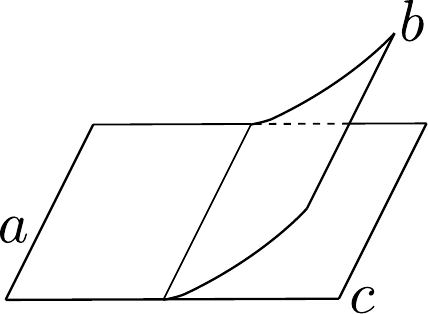}
    \caption{The weights $a,b,c$ must satisfy the equality $a=b+c$.}
     
\end{figure}
The way $\Sigma$ induces a weight system is the following: we fix a point in the interior of each sector and we assign to each sector the number of intersections between $\Sigma$ and the fiber of $N_B$ over this fixed point.
We denote with $\omega_i$'s the weights of the half disc sectors $D_i$'s. 

We have already observed that the other sectors of $B$ are organised so that each of the annuli $A_1,\dots, A_k$ contains 
$k|a_0|$ among discs and half discs sectors. When $a_0>0$ (resp. $a_0<0$), for each $i=1,\dots, k$, we order in each annulus these sectors by following the direction of $\alpha_i$ (resp. $-\alpha_i$) and denote with $\Delta_{i,j}$ the discs contained in the annulus $A_i$, with $j=1,\dots, ka_0$. We denote the weight of the disc $\Delta_{i,j}$ with $\omega_{i,j}$. See Figure \ref{disks in annulus} for an example.

\begin{figure}[H]
   \centering
    \includegraphics[width=0.8\textwidth]{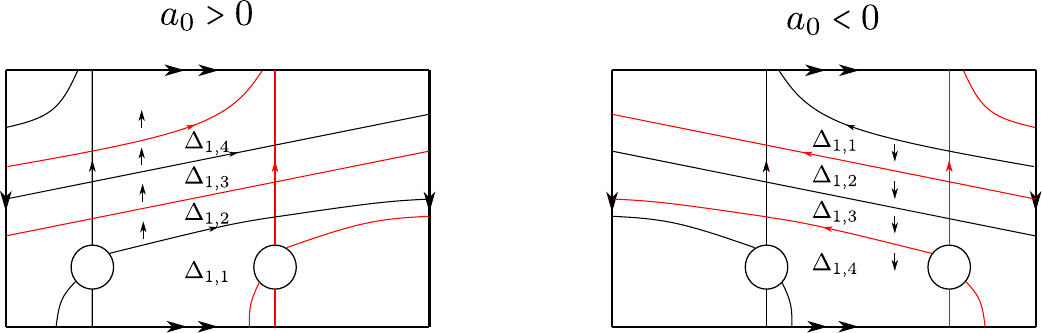}
    \caption{In this example it is showed how to order the discs inside the annulus $A_1$ in the case $a_0>0$ and in the case $a_0<0$.}
    \label{disks in annulus}
\end{figure}

Let us fix an annulus $A_{i_0}$. Each pair of consecutive discs $\Delta_{i_0,j}, \Delta_{i_0,j+1}$ is separated by a subarc of $h(\alpha_l)$, for some $l=1,\dots, k$. Since the arcs $\alpha_i$'s (and therefore also the $h(\alpha_i)$'s) are parallel it follows by the orientation of the discs $D_i$'s that we have:
$$
\omega_{i_0,j}+\omega_{l}=\omega_{i_0,j+1}.
$$
This implies the following chain of inequality:
$$
\omega_{i_0,1}\leq \omega_{i_0,2}\leq \dots\leq \omega_{i_0,k|a_0|}\leq \omega_{i_0,1}.
$$
Therefore the weights $\omega_{i_0,j}$ are all equal and 
since all the arcs $h(\alpha_i)$'s intersect the annulus $A_{i_0}$ we have that $\omega_l=0$ for each $l=1,\dots, k$. Therefore the sectors of $B$ contained in $S$ all have the same weight and the discs $D_i$'s have weight zero. This means that $\Sigma$ is a finite number of parallel copies of $S$.
\end{itemize}
By construction, $\partial M_h\setminus \partial B$ is a union of bigons (see Figure \ref{boundary train tracks2}). Moreover $B$ does not carry any closed surface, and therefore it does not carry tori bounding a solid torus in $M_h(s_1,\dots, s_k)$, for any multislope $(s_1,\dots, s_k)$. Since $M_h$ is irreducible and its boundary is union of incompressible tori, the hypotheses of Theorem \ref{boundary train tracks} are fulfilled.
\end{proof}

\begin{prop}\label{teorema foliazioni part 1}
If $a_0>0$, for any multislope $(s_1,\dots, s_k)\in (\infty, 1)^k$ the branched surface $B$ fully carries an essential lamination intersecting the boundary of $M_h$ in parallel curves of multislope $(s_1,\dots, s_k)$. 
If $a_0<0$ the same happens for any multislope $(s_1,\dots, s_k)\in (-1,\infty)^k$. 
\end{prop}
\begin{proof}
We study the multislopes realised by the boundary train tracks of $B$.
In order to do this we assign rational weight systems to our boundary train tracks. Since our train tracks are oriented, we can associate to such a weight system the rational number $\sfrac{w_{\mu_i}}{w_{\lambda_i}}$, where $w_{\mu_i}$ and $w_{\lambda_i}$ are the \emph{weighted} intersections of the train tracks with our fixed meridians $\mu_i$ and longitudes $\lambda_i$, as we would do with oriented simple closed curves. This quotient can be interpreted as a \emph{slope} in the $i$-th boundary component of $M_h$. In fact it is can be shown that each slope $\sfrac{p}{q}$ obtained in this way is realised by the train track. Since we want to study slopes \emph{fully} carried by these train tracks, we have to require that each weight is strictly positive: if the weight of an arc is zero, the associated slope will not intersect the fibers over that arc. For details, see \cite{PH}.

The boundary train tracks of $B$ are all the same for each boundary tori, and only depend on the sign of $a_0$. The two possible types of boundary train tracks are depicted in Figure \ref{boundary train tracks2}.

\begin{figure}[H]
   \centering
    \includegraphics[width=0.65\textwidth]{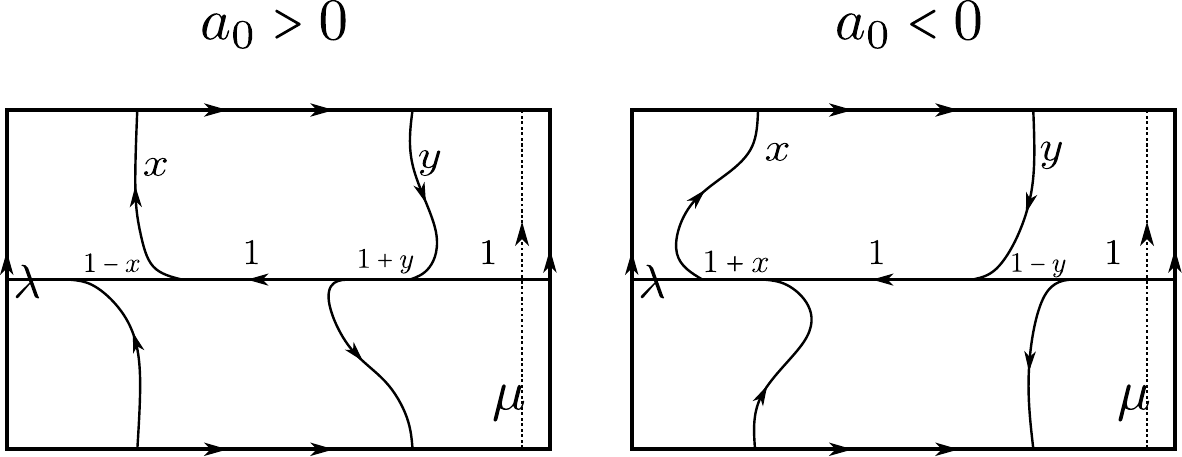}
    \caption{The two possible boundary train tracks with weight systems.}
    \label{boundary train tracks2}
\end{figure}
We also endowed the two train tracks with weight systems. The slopes of these weight systems are always $x-y$, but since we have to impose that each sector of the train tracks has positive weight we have that:
\begin{itemize}
    \item if $a_0>0$, $x$ can vary in $(0,1)$ and $y$ can vary  in $(0,+\infty)$;
    \item if $a_0<0$, $x$ can vary in $(0,+\infty)$  and $y$ can vary in $(0,1)$.
\end{itemize}

By letting $x,y$ vary we have that when $a_0>0$ the boundary train tracks realise all multislopes in $(\infty, 1)^k$ and when $a_0<0$ the boundary train tracks realise all slopes in $(-1,\infty)^k$. Thanks to Lemma \ref{B is laminar} we can apply Theorem \ref{boundary train tracks} to obtain the desired essential laminations.
\end{proof}

\begin{lemma}\label{extending laminations}
All the laminations constructed in the previous proposition extend to taut foliations of the filled manifolds.
\end{lemma}
\begin{proof}
Let $\Lambda$ be one of the laminations constructed in Proposition \ref{teorema foliazioni part 1} and suppose that $\Lambda$ intersect $\partial M_h$ in parallel curves of multislope $(s_1,\dots, s_k)$. We can also suppose that $\partial_h N_B\subset \Lambda$.
First of all we notice that if the multislope is different from $(0,\dots, 0)$ then $\Lambda$ does not have any compact leaves.  In fact any leaf of $\Lambda$ is carried by $B$ and we have showed in the proof of Lemma \ref{B is laminar} that the only connected compact surface carried by $B$ is the $k$-holed torus $S$. Therefore if $\Lambda$ has a compact leaf then it should intersect $\partial M_h$ in parallel curves of multislope $(0,\dots,0)$. 

We now consider the abstract closures (in a path metric on $M_h$) of the complementary regions of $\Lambda$. These closures are $[0,1]$-bundles; in fact they are unions, along $\partial_v N_B$, of:

\begin{itemize}
    \item components of $M_h\setminus {\rm int}(N_B)$, that are products of the type $F\times [0,1]$, where $F$ is a surface, with 
    $$
    \partial_h N_B\cap (F\times [0,1])=F\times \{0,1\}
    $$
    and 
    $$
    \partial_v N_B\cap (F\times [0,1])=\partial' F\times [0,1]
    $$
    where $\partial'F$ is the closure of $\partial F \setminus \partial M_h$;
    \item abstract closures of the components of $N_B\setminus{\Lambda}$. Since $\Lambda$ intersects transversely the fibers of $N_B$ also these closures are products with the same properties of the components of $M_h\setminus{\rm int}(N_B)$.
\end{itemize}
Each component of the vertical boundary of $N_B$ is an annulus $S^1\times [0,1]$ or a disc $[0,1]\times [0,1]$, where each interval $\{*\}\times [0,1]$ is contained in a  fiber of $N_B$. Both the product structures of the components of $M_h\setminus {\rm int}(N_B)$ and of the abstract closures of the components of $N_B\setminus \Lambda$ define a foliation of the vertical boundary transverse to the interval fibers. Any of two such foliations are isotopic and therefore also the abstract closures of the complementary regions of $\Lambda$ are products.

In particular, since the horizontal boundary of the closures of these complementary regions are leaves of $\Lambda$ we can foliate these bundles with parallel leaves to obtain a foliation of the whole $M_h$. This foliation has no compact leaves and intersects the boundary of $M$ in parallel curves of multislope $(s_1,\dots,s_k)$. Therefore the leaves of this foliation can be capped with the meridional discs of the solid tori to obtain a foliation of the filled manifold $M_h(s_1,\dots,s_k)$ that has no compact leaves as well, and that is therefore  taut (see \cite[Example~4.23.]{Calegari}).
\end{proof}

\begin{remark}\label{remark kalelkar roberts}
In the terminology of \cite{KR1}, if $|a_0|=1$ then the pair of parallel $k$-uples 
$$
(\widetilde{h(\alpha)},\alpha)
$$
is good and oriented, where $\alpha=(\alpha_1,\dots,\alpha_k)$ and $\widetilde{h(\alpha)}=(\widetilde{h(\alpha_1)},\dots,\widetilde{h(\alpha_k)})$. In this case the branched surface constructed in the previous discussion coincides with the branched surface associated to the sequence $(\smash{\widetilde{h(\alpha)}},\alpha)$ by Kalelkar and Roberts in \cite{KR1}.
\end{remark}

\subsubsection{Proof of the second part of Theorem \ref{teorema foliazioni}}
We now focus our attention on the second part of Theorem \ref{teorema foliazioni} and we define a new branched surface. We fix a new set of arcs $\alpha_1,\dots, \alpha_k$ in the following way. We consider arcs $\beta_1,\dots,\beta_k$ as in Figure \ref{horizontal arcs} and choose $\alpha_i$ so that $\smash{\widetilde{h(\alpha_i)}}=\beta_i$. One example is depicted in Figure \ref{orientations}.

\begin{figure}[H]
   \centering
    \includegraphics[width=0.5\textwidth]{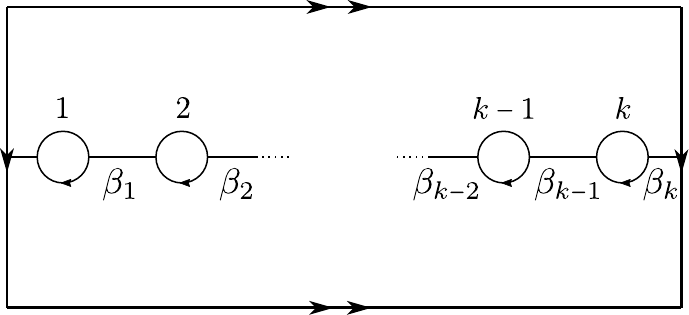}
    \caption{The arcs $\beta_i$'s.}
    \label{horizontal arcs}
\end{figure}

We now give orientations to the arcs $\alpha_i$'s in order to build our branched surfaces. It will be simpler to state how to assign orientations to the $\beta_i$'s and we will orient each $\alpha_i$ as isotopic to $h^{-1}(\beta_i)$, for $i=1,\dots,k$.

\begin{definition}
We  say that an orientation of the arcs $\beta_i$'s is \textbf{coherent} if the following hold:
\begin{itemize}
    \item if a boundary component $\partial_i S$ has label $p_+$ or $p_-$ the arcs $\beta_{i-1}$ and $\beta_{i}$ intersecting $\partial_i S$ are oriented so that the first starts at $\partial_i S$ and the second ends at $\partial_i S$, or viceversa. In this case we say that $\beta_{i-1}$ and $\beta_{i}$ \emph{have the same direction};
    \item if a boundary component $\partial_i S$ has label $n$ the arcs $\beta_{i-1}$ and $\beta_{i}$ intersecting $\partial_i S$ are oriented so that both start or both end at $\partial_i S$. In this case we say that $\beta_{i-1}$ and $\beta_{i}$ \emph{have opposite directions}. In case the arcs both start at $\partial_i S$ we say that the component is of type $n_{\textbf{o}}$ (the subscript \textbf{o} stands for ``out'') and if both end at $\partial_i S$ we say that is of type $n_{\textbf{i}}$ (\textbf{i} standing for ``in'').
\end{itemize}
\end{definition}
See Figure \ref{orientations} for an example\footnote{recall that the factorisation of the monodromy $h$ is to be read from right to left; this should help to figure out why $h(\alpha_i)=\beta_i$.}.
Notice that there is always an even number of boundary components of $S$ with label $n$. Moreover the boundary components with label $n$ are alternately of type $n_{\textbf{o}}$ and $n_{\textbf{i}}$. 
\begin{figure}[H]
   \centering
    \includegraphics[width=0.6\textwidth]{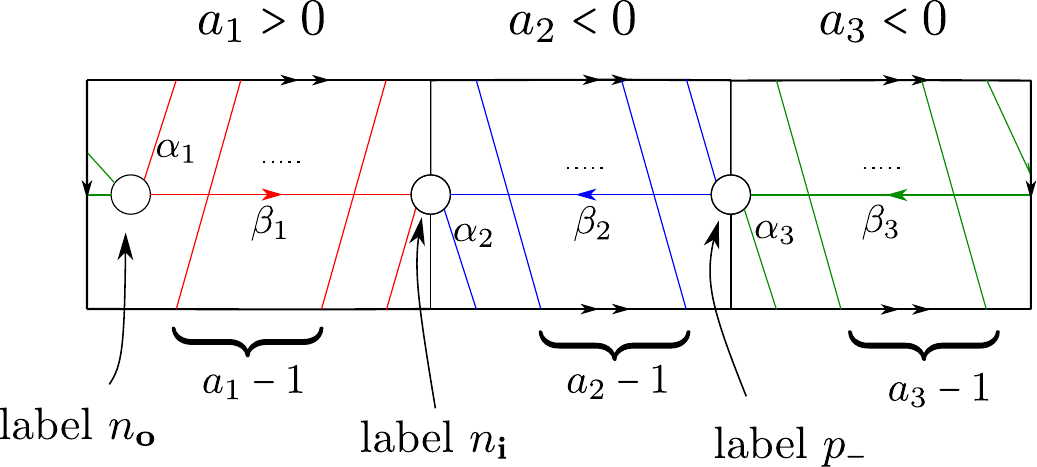}
    \caption{An example of a coherent orientation. In this case the arcs $\beta_2$ and $\beta_3$ have the same direction, while the arcs $\beta_1$ and $\beta_2$ and the arcs $\beta_3$ and $\beta_1$ have opposite directions.}
    \label{orientations}
\end{figure}
We will soon use coherent orientations to build branched surfaces. First of all we prove the following lemma.

\begin{lemma}
There always exist exactly two different coherent orientations of the arcs $\beta_i$'s.
\end{lemma}

\begin{proof}
We fix an orientation of the arc $\beta_1$. We prove that there exists a unique coherent orientation of the arcs $\beta_i$'s agreeing with the fixed orientation on $\beta_1$ and this implies the thesis. We orient the arcs $\beta_i$'s inductively. Suppose that we have oriented $\beta_1,\dots, \beta_j$. Then:
\begin{itemize}
    \item if $\partial_{j+1}S$ has label $p_+$ or $p_-$ we orient $\beta_{j+1}$ so that it has the same direction of $\beta_j$;
    
    \item if $\partial_{j+1}S$ has label $n$ we orient $\beta_{j+1}$ so that its direction is opposite to the one of $\beta_j$.
\end{itemize}

In other words, once we have fixed an orientation on $\beta_1$ the coherence condition completely determines the orientations of $\beta_2,\dots, \beta_k$. The only thing to be checked in order to prove that this orientation is actually coherent is the behaviour of $\beta_k$ and $\beta_1$ at $\partial_1 S$. Since there is always an even number of boundary components of $S$ with label $n$ it follows that:
\begin{itemize}
    \item if $\partial_1 S$ has label $p_+$ or $p_-$ then the direction changes an even number of times between $\beta_1$ and $\beta_k$ and therefore $\beta_1$ and $\beta_k$ have the same direction;
    \item if $\partial_1 S$ has label $n$ then the direction changes an odd number of times between $\beta_1$ and $\beta_k$ and therefore $\beta_1$ and $\beta_k$ have opposite directions.
\end{itemize}
Therefore the orientation defined in this way is coherent and this concludes the proof.
\end{proof}

We fix a coherent orientation and as usual we consider the branched surface $B$ that is the union of $S$ and the images in $M_h$ of the discs  $\alpha_i\times [0,1]\subset S\times [0,1]$. We denote the image of $\alpha_i\times [0,1]$ with $D_i$ and orient the discs $D_i$'s so that the orientation on their boundary induces the given orientation on the $\alpha_i$'s. Exactly as before we have:

\begin{lemma}\label{B is laminar_2}
The branched surface $B$ is laminar and satisfies the hypotheses of Theorem \ref{boundary train tracks}.
\end{lemma}

\begin{proof}
The proof is analogous to the one of Lemma \ref{B is laminar}. We only need to prove that $B$ contains no sink discs or half sink discs, and that the only connected surface properly embedded in $M_h$ carried by $B$ is $S$.
\begin{itemize}
\item \emph{B contains no sink discs or half sink discs:} there are $k$ sectors of $B$ that coincide with the discs $D_i$'s and they always have cusp directions pointing outside. We focus our attention on the sectors contained in $S$. We consider the $k$ annuli $A_i$ obtained by cutting $S$ along the arcs of Figure \ref{parallel arcs}. Each of these annuli contains in its interior some disc and half disc sectors and intersects two other half disc sectors. The former are never sink because each of these sectors has in its boundary two parallel subarcs of some arc of the $\alpha_i$'s, as for example Figure \ref{annulus} shows.

\begin{figure}[H]
   \centering
    \includegraphics[width=0.23\textwidth]{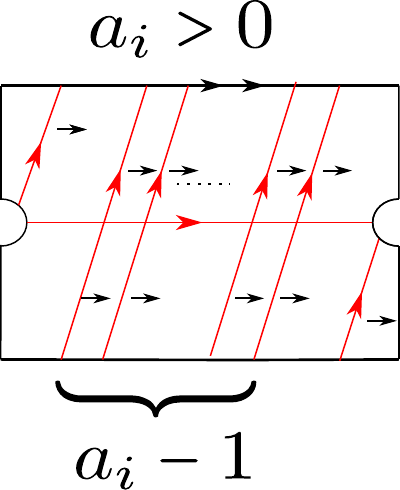}
    \caption{The annulus $A_i$.}
    \label{annulus}
\end{figure}

We now claim the following:

\textbf{Claim: }since we have fixed a coherent orientation of the arcs $\beta_i$'s, the cusp directions along the arcs $\alpha_i$'s all point in the same direction.

The claim implies that the sectors belonging to two consecutives annuli are never sink because each of these sectors has in its boundary two subarcs of two consecutive arcs of the $\alpha_i$'s. For an example, see Figure \ref{annulus2}.

\textbf{Proof of the claim: }
We first notice that when $a_i>0$ (resp. $a_i<0$) the cusp direction along the arc $\alpha_i$ has the same (resp. opposite) direction of $\beta_i$ (recall that the cusp direction always points to the right along the oriented arcs $\alpha_i$'s). Therefore to prove the claim it is sufficient to prove that $\beta_i$ and $\beta_j$ have the same direction if and only if $a_ia_j>0$, and to prove this it is enough to prove that $a_1a_i>0$ if and only if $\beta_1$ and $\beta_i$ have the same direction. We prove this by induction on $i$. 
If $i=2$ this follows from the definition of coherent orientation. We suppose now that the thesis is true for $i$ and we prove it for $i+1$. Suppose that $a_1a_{i+1}>0$; then if $a_1a_i>0$ we know by inductive hypothesis that $\beta_1$ and $\beta_i$ have the same direction. Moreover we deduce that $a_ia_{i+1}>0$ and by the definition of coherent orientation that $\beta_i$ and $\beta_{i+1}$ have the same direction and therefore also $\beta_1$ and $\beta_{i+1}$ have the same direction. The other cases can be analysed similarly. This concludes the proof of the claim.

\begin{figure}[H]
   \centering
    \includegraphics[width=0.9\textwidth]{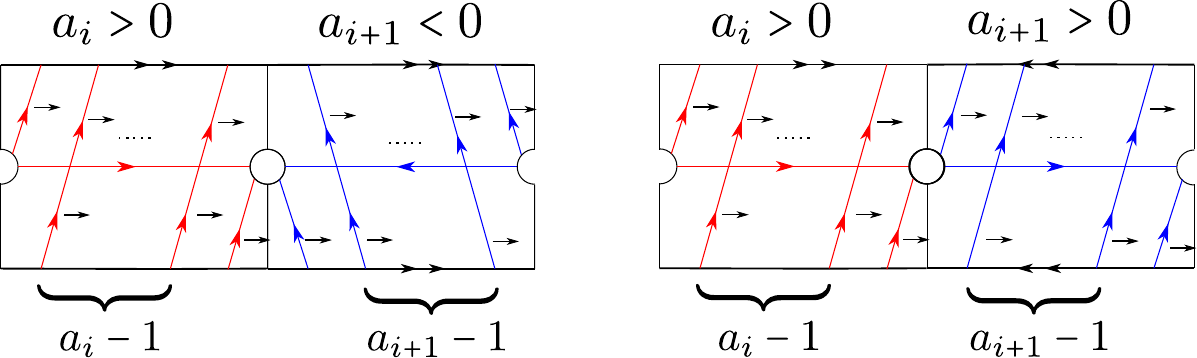}
    \caption{The figure shows how the choice of a coherent orientation implies that the cusp directions along the arcs $\alpha_i$'s all have the same direction.}
    \label{annulus2}
\end{figure}

\item \emph{ The only connected compact surface properly embedded in $M_h$ carried by $B$ is $S$:} suppose that $\Sigma$ is a compact surface carried by $B$. $\Sigma$ induces an integral weight system on $B$. We denote with $\omega_i$ the weight of the discs $D_i$. The number of sectors contained in $S$ is equal to $N=\sum_{i=1}^k|a_i|$. Since we have fixed a coherent orientation of the arcs $\beta_i$, the cusp directions along the arcs $\alpha_i$'s all point in the same direction. We order the sectors in $S$ according to this direction as depicted in Figure \ref{disks in sectors}; we denote them with $\Delta_{l}$ and we denote their weights with $\delta_{l}$, where $1\leq l\leq N$. 

\begin{figure}[H]
   \centering
    \includegraphics[width=0.5\textwidth]{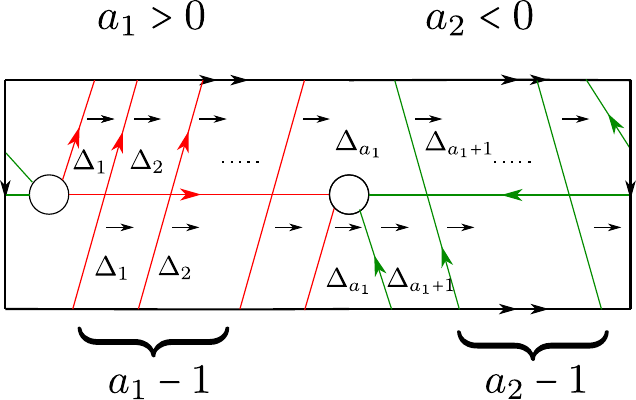}
    \caption{An example that shows how to label the sectors $\Delta_l$'s.}
    \label{disks in sectors}
\end{figure}

As in the proof of the first part of the theorem, we have that 
$$
\delta_1\leq \delta_2 \leq \cdots\leq \delta_N\leq\delta_1.
$$
Therefore these weights are all equal and this implies that the discs $D_i$'s all have weight zero; that is to say, $\Sigma$ is a finite number of parallel copies of $S$.
\end{itemize}

\end{proof}

\begin{prop}\label{teorema foliazioni part 2}
Let $I_1,\dots, I_k$ and $J_1,\dots,J_k$ be the intervals defined in the discussion before the statement of Theorem \ref{teorema foliazioni} and let $B$ denote the branched surface associated to a coherent orientation of the arcs $\beta_i$'s. Then for one choice of coherent orientation of the arcs $\beta_i$'s, $B$ fully carries essential laminations intersecting the boundary of $M_h$ in parallel curves of multislope $(s_1,\dots, s_k)$, for $(s_1,\dots, s_k)\in I_1\times\cdots \times I_k$. Choosing the other coherent orientation yields $B$ that fully carries essential laminations intersecting the boundary of $M_h$ in parallel curves of multislope $(s_1,\dots, s_k)$, for $(s_1,\dots, s_k)\in J_1\times\cdots \times J_k$. Moreover these laminations extend to taut foliations of the filled manifold $M_h(s_1,\dots, s_k)$.
\end{prop}

\begin{proof}
We focus our attention on the boundary train tracks of $B$. For a fixed boundary component of $S$  we have the four possible configurations showed in Figure \ref{boundary label} and for each of this configurations we have two possible way to fix a coherent orientation.

\begin{figure}[H]
   \centering
    \includegraphics[width=0.9\textwidth]{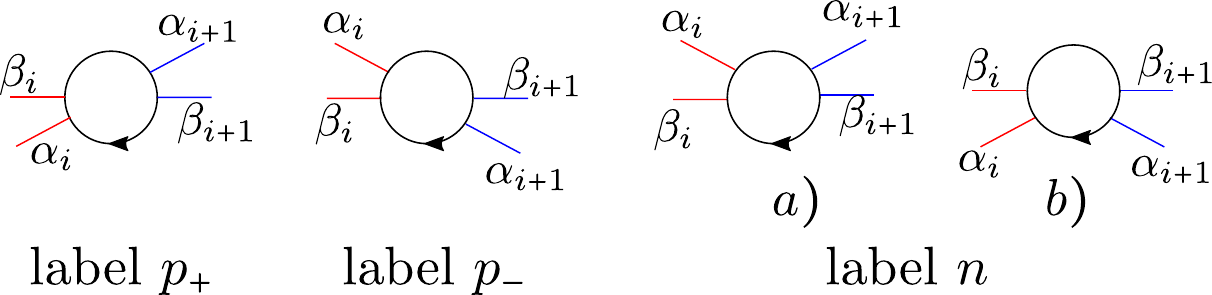}
    \caption{The four possible configuration of arcs in a neighbourhood of $\partial_{i+1} S$.}
    \label{boundary label}
\end{figure}
If the boundary component has label $p_+$ or $p_-$ the type of the boundary train track does not depend on the choice of the coherent orientation, and is described in the following figure, where for concreteness we have assigned an orientation to the arcs, and where in the middle picture we have also described the branched surface in a neighbourhood of the boundary component. If we consider the other coherent orientation, we obtain the same train tracks.
\begin{figure}[H]
   \centering
    \includegraphics[width=0.9\textwidth]{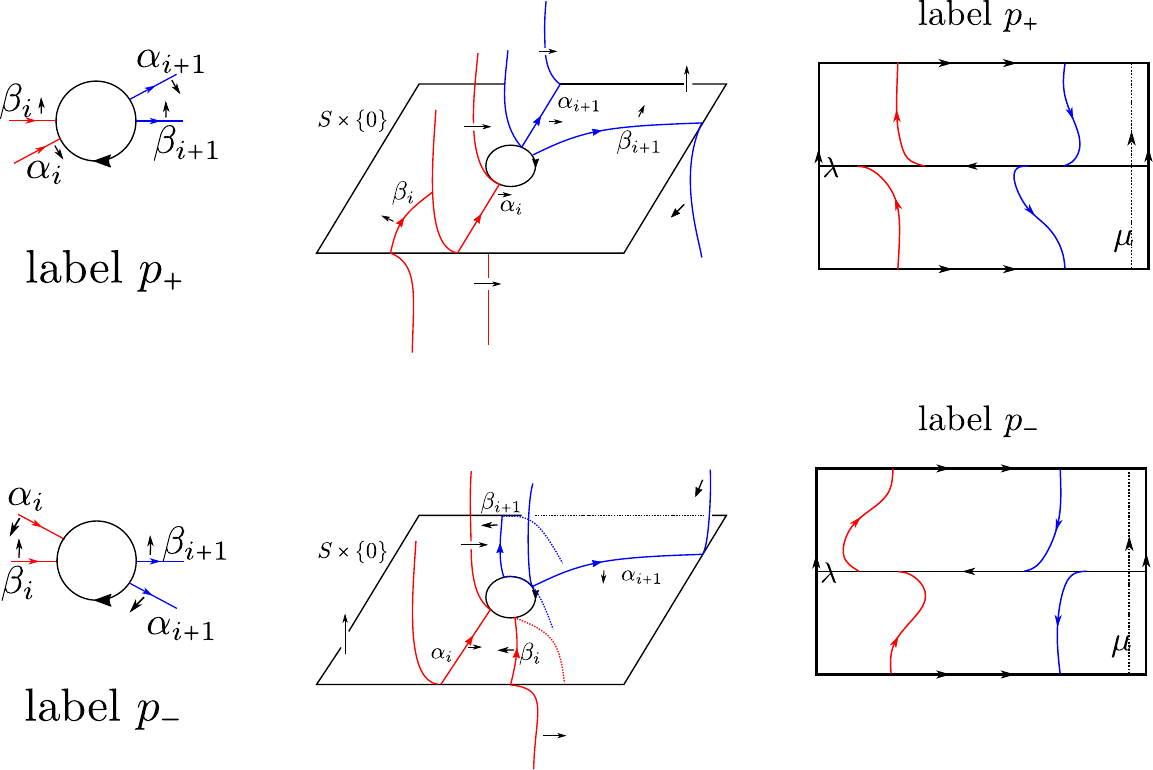}
    \caption{In this figure we describe the branched surface in a neighbourhood of the $(i+1)$-th boundary component of $M_h$ and its boundary train track in the case of label $p_+$ and $p_-$.}
\end{figure}
By assigning weights to these train tracks as we have already done in the proof of Proposition \ref{teorema foliazioni part 1} we have that if the label is $p_+$ the train track realises all the slopes in the interval $(\infty,1)$, while if the label is $p_-$ the slopes realised are those in the interval $(-1,\infty)$.

On the other hand if the boundary component has label $n$ the choice of the orientation yields two different train tracks. We represent the possible train tracks in Figure \ref{boundary label n}. Notice that the train tracks depend only on the orientation of the arcs, and not on the label $a)$ or $b)$ of the configuration.

\begin{figure}[H]
   \centering
    \includegraphics[width=0.65\textwidth]{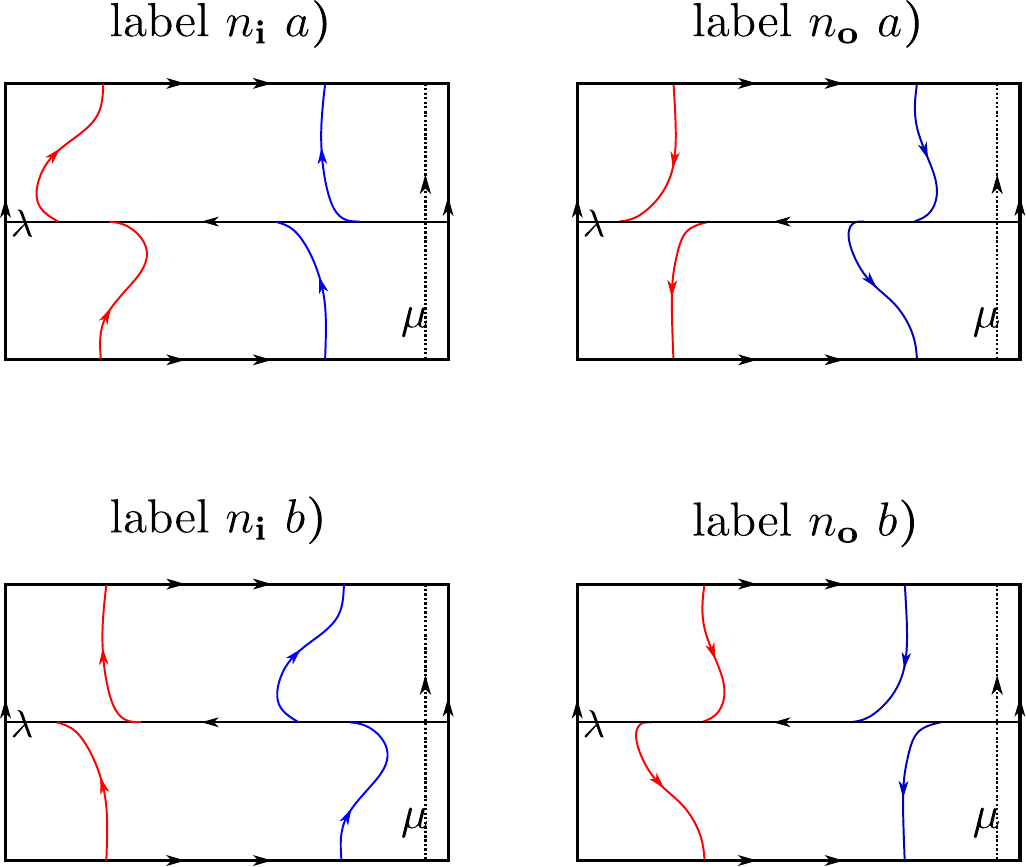}
    \caption{The possible boundary train tracks associated to a boundary component with label $n$. Notice that the train tracks depend only on the orientation of the arcs, and not on the label $a)$ or $b)$ of the configuration.}
    \label{boundary label n}
\end{figure}
The train tracks on the left realise all the slopes in the interval $(0,\infty)$, while those on the right realise the slopes in the interval $(\infty, 0)$. 

By fixing one or the other of the two possible coherent orientations, we have that the boundary train tracks of $B$ realise all the multislopes in $I_1\times\dots \times I_k$ and $J_1\times\dots \times J_k$.
By virtue of Lemma \ref{B is laminar_2}, we can apply Theorem \ref{boundary train tracks} to obtain the desired essential laminations and Lemma \ref{extending laminations} implies that these laminations extends to taut foliations of the filled manifolds.
\end{proof}

\begin{proof}[Proof of Theorem \ref{teorema foliazioni}]
The first part of the theorem is the content of Proposition \ref{teorema foliazioni part 1} and Lemma \ref{extending laminations}. The second part is the content of Proposition \ref{teorema foliazioni part 2}.
\end{proof}

\begin{ex}
For each natural number $n$ we consider the $n$-component oriented link $\mathcal{L}_n$ in figure. 

\begin{figure}[H]
   \centering
    \includegraphics[width=0.5\textwidth]{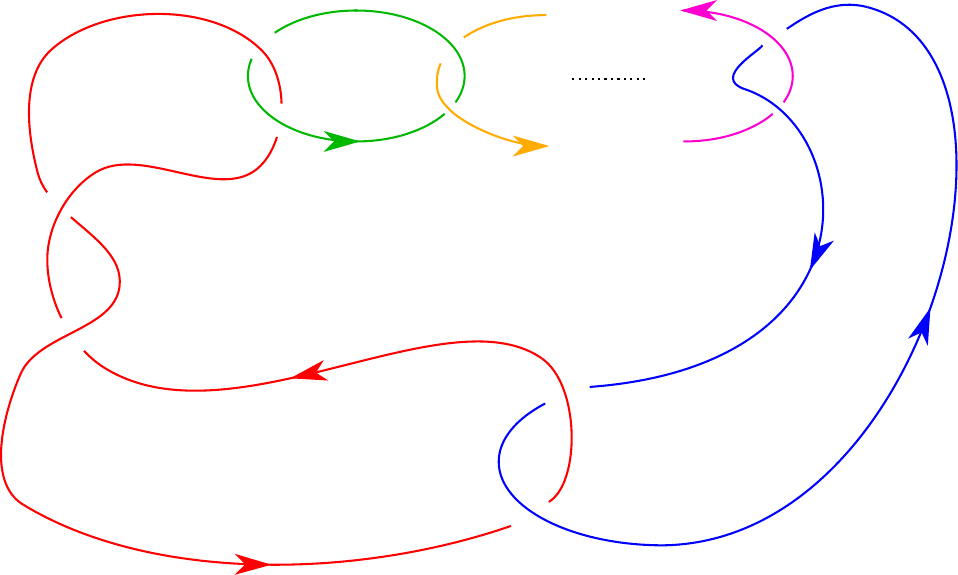}
    \caption{The link $\mathcal{L}_n$.}
    \label{Ln_before}
\end{figure}
We can represent the link $\mathcal{L}_n$ in a different way, as in Figure \ref{L_n}. With this representation, it is evident that $\mathcal{L}_n$ can be realised as a plumbing of Hopf bands. Therefore (see \cite{G1,S1}) $\mathcal{L}_n$ is a fibered link, with fiber surface a torus with $n$ open discs removed, and the monodromy associated to this fiber is 
$$
h=\tau_0^{-1}\tau_1\dots\tau_n
$$
where $\tau_i$ is the positive Dehn twist along the curve $\gamma_i$. 
\begin{figure}[H]
   \centering
    \includegraphics[width=0.6\textwidth]{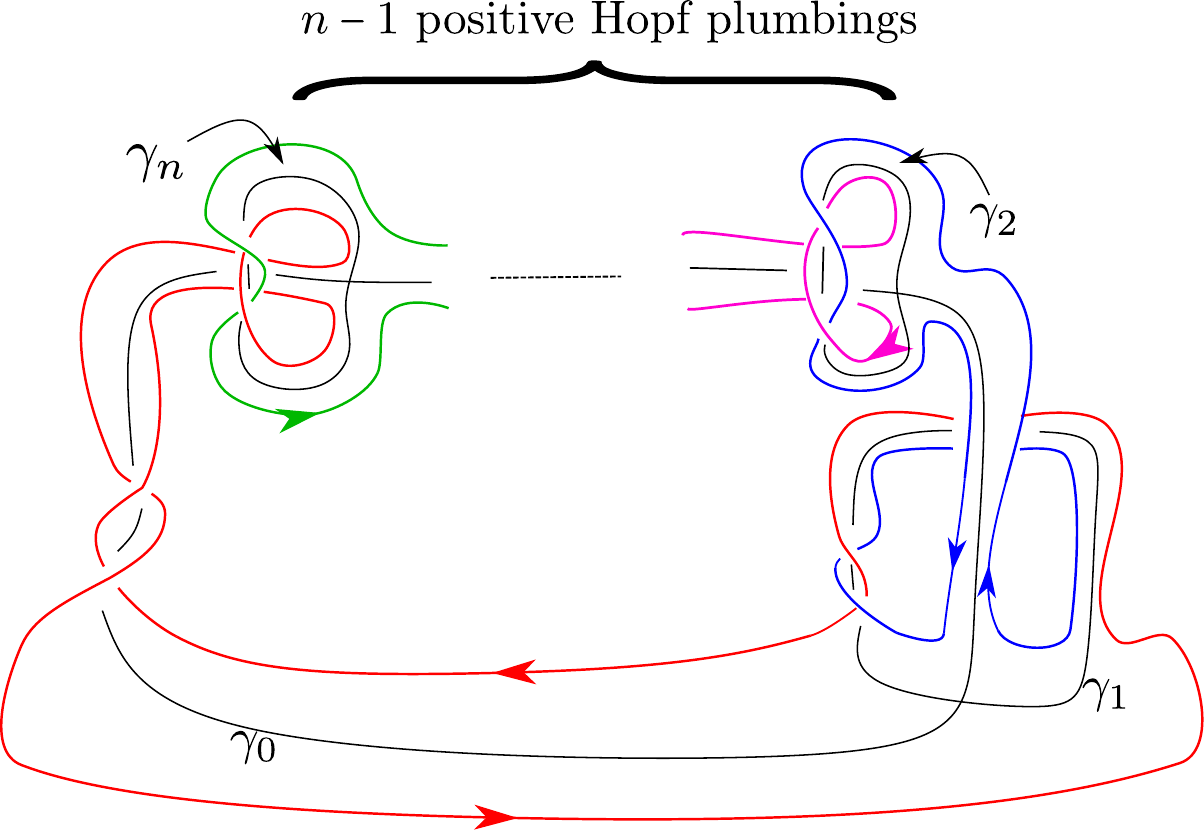}
    \caption{A description of $\mathcal{L}_n$ as a plumbing of Hopf bands.}
    \label{L_n}
\end{figure}

\begin{figure}[H]
   \centering
    \includegraphics[width=0.6\textwidth]{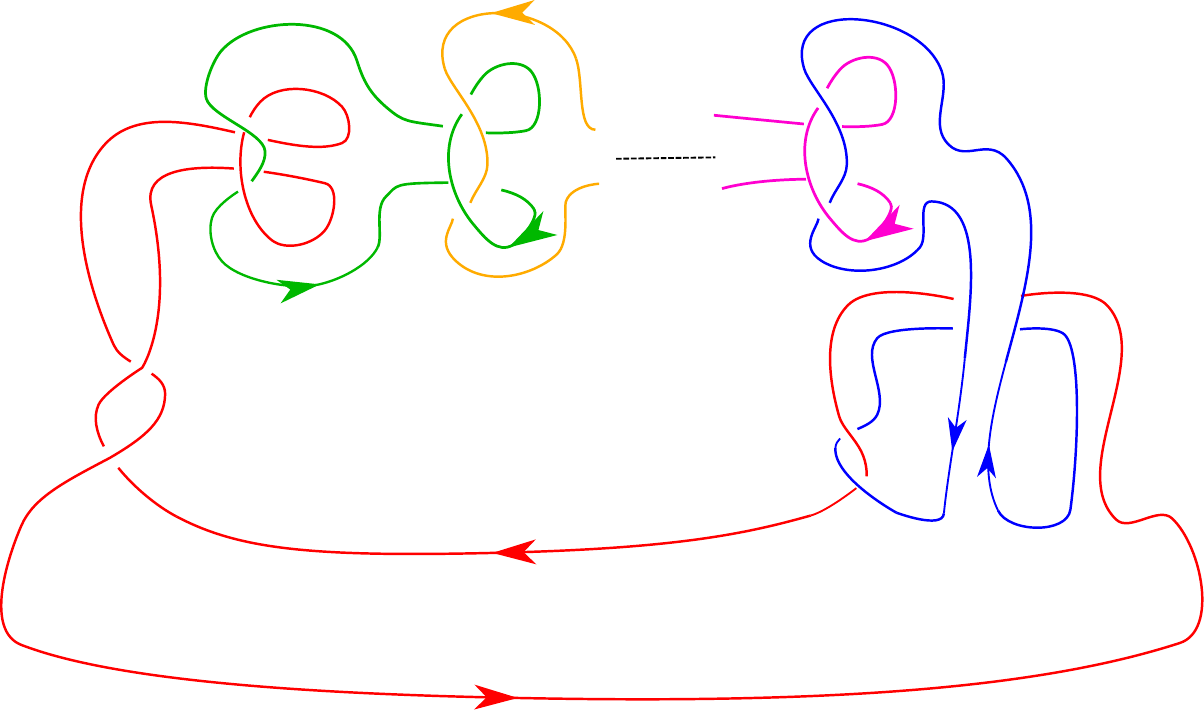}
    \caption{Another picture of the link $\mathcal{L}_n$. Without the cores of the Hopf bands it may be easier to see the isotopy to the link depicted in Figure \ref{Ln_before}.}
    \label{L_n without cores}
\end{figure}

We now prove that $\mathcal{L}_n$ is a hyperbolic link. We recall the following theorem of Penner \cite{P}: 
\begin{theorem}(\cite {P})
Suppose that $\mathcal{C}$ and $\mathcal{D}$ are each disjointly embedded collections of essential simple closed curves (with no parallel components) in an oriented surface
$F$ so that $\mathcal{C}$ hits $\mathcal{D}$ efficiently and $\mathcal{C}\cup \mathcal{D}$ fills $F$. Let $R(\mathcal{C}^+,\mathcal{D}^-)$ be the free
semigroup generated by the Dehn twists $\{\tau_c^{+1}: c\in \mathcal{C}\} \cup \{\tau_d^{-1}: d\in\mathcal{D}\}$. Each
component map of the isotopy class of $w\in R(\mathcal{C}^+ ,\mathcal{R}^-)$ is either the identity or
pseudo-Anosov, and the isotopy class of $w$ is itself pseudo-Anosov if each $\tau_c^{+1}$ and $\tau_d^{-1}$ occur at least once in $w$.
\end{theorem}
In the statement of the previous theorem, ``$\mathcal{C}\cup \mathcal{D}$ \emph{fills} $F$'' means that each component of the complement of $\mathcal{C}\cup \mathcal{D}$ is a disc, a boundary-parallel annulus, or a puncture-parallel punctured disc. Moreover ``\emph{$\mathcal{C}$ hits $\mathcal{D}$ efficiently}'' if there is no bigon in $F$ with boundary made of one arc of a curve $c\in \mathcal{C}$ and one arc of a curve $d\in \mathcal{D}$.

In our case we set $\mathcal{C}=\{\gamma_1,\dots, \gamma_n\}$ and $\mathcal{D}=\{\gamma_0\}$ and we can apply this theorem to deduce that the monodromy associated to $\mathcal{L}_n$ is a pseudo-Anosov map; applying Thurston \cite{Th2} we deduce these links are hyperbolic.

Moreover Theorem \ref{teorema foliazioni} applies to these links and we can deduce that for any multislope $(s_1,\dots,s_n)\in (\infty,1)^n \cup (-1,\infty)^n$, the filling of the exterior of $\mathcal{L}_n$ with multislope $(s_1,\dots,s_n)$ supports a coorientable taut foliation. Recall that these slopes are referred to the meridian-longitude bases given by the mapping torus; since the components of the link $\mathcal{L}_n$ do not have pairwise linking number zero, the longitudes of these bases do not coincide with the canonical longitudes of the link.
\end{ex}

\subsection{The Whitehead link case}\label{The Whitehead link case}

We now return to the Whitehead link exterior. Recall from the discussion preceding Figure \ref{hopf plumbing} that the Whitehead link is a fibered link with fiber surface a $2$-holed torus and monodromy $h=\tau_0\tau_1\tau_2^{-1}$, where the curves $\gamma_i$ are represented in Figure \ref{Abstract torus}. In what follow we will identify the exterior of the Whitehead link with the mapping torus $M_h$.

\begin{figure}[H]
   \centering
    \includegraphics[width=0.4\textwidth]{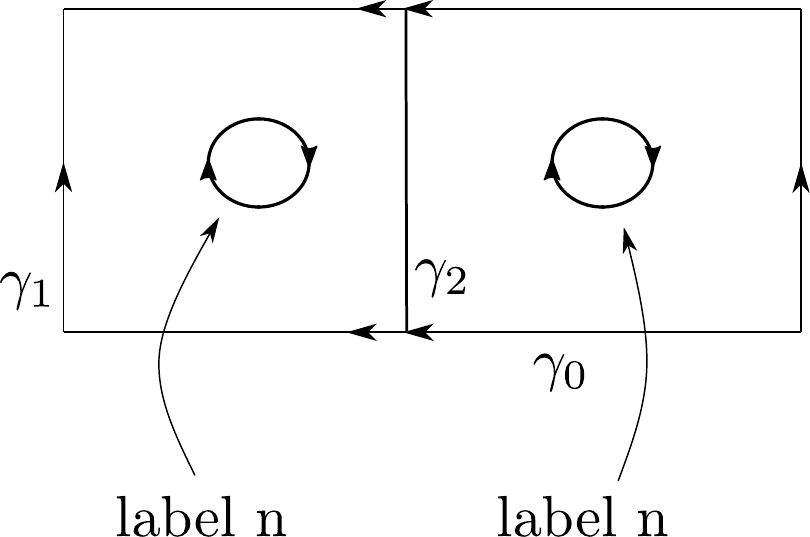}
    \caption{An abstract description of the fiber surface for the Whitehead link. The monodromy is given by $h=\tau_0\tau_1\tau_2^{-1}$. We have also indicated the labels of the boundary components of the torus.}
    \label{Abstract torus}
\end{figure}

As a consequence of Theorem \ref{teorema foliazioni} and of the fact that the components of the Whitehead link have linking number zero, we have

\begin{cor}\label{corollario quasi tutto}
Let $\left(\sfrac{p_1}{q_1},\sfrac{p_2}{q_2}\right)$ be a multislope in

$$
(\infty,1)^2 \quad \cup\quad  (0,\infty)\times(\infty,0)\quad \cup\quad (\infty,0)\times(0,\infty).
$$
Then $S^3_{\sfrac{p_1}{q_1},\sfrac{p_2}{q_2}}({\rm WL})$ supports a coorientable taut foliation.
\end{cor}

We have been able to prove that for slopes in the region depicted in Figure \ref{foliations}, the corresponding filling on the Whitehead exterior supports a coorientable taut foliation.

\begin{figure}[H]
   \centering
    \includegraphics[width=0.45\textwidth]{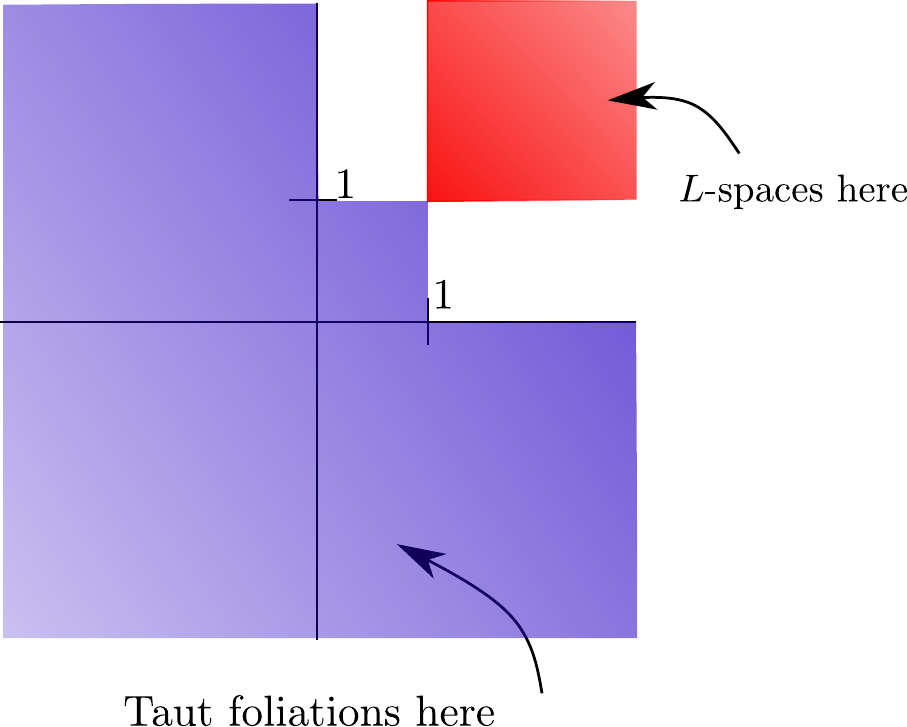}
    \caption{The figure describes what we have been able to prove up to now. The blue points are the slopes whose corrisponding filling supports a coorientable taut foliation; the red points are those whose corresponding filling is an $L$-space.}
    \label{foliations}
\end{figure}

We now cover the remaining regions of Figure \ref{foliations}. We define a branched surface by considering the pair of arcs $\alpha,\beta$ shown in Figure \ref{B_2}-(left). We apply the monodromy $h$ and we obtain what is depicted in Figure \ref{B_2}.

\begin{figure}[H]
   \centering
    \includegraphics[width=1\textwidth]{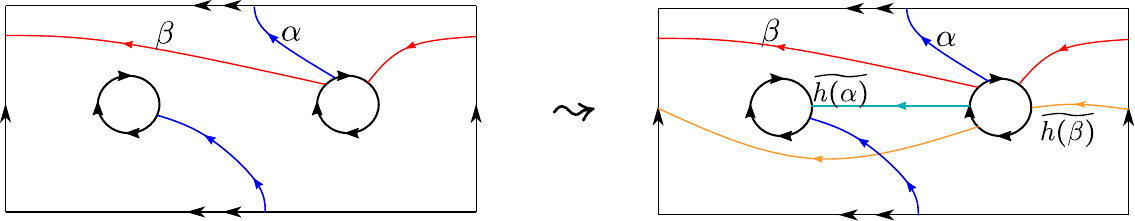}
    \caption{The arcs $\alpha$ and $\beta$ and their (perturbed) images via $h$.}
    \label{B_2}
\end{figure}

As usual, we consider the branched surface $B$ associated to the arcs $\alpha$ and $\beta$.

\begin{lemma}
The branched surface $B$ is laminar and satisfies the hypotheses of Theorem \ref{boundary train tracks}.
\end{lemma}
\begin{proof}
Since $S\setminus(\alpha \cup \beta)$ has no disc components, by virtue of Lemma \ref{lemma laminar} we only need to prove that $B$ contains no sink disc or half sink discs and this is showed in Figure \ref{B_2 no sink}.

\begin{figure}[H]
   \centering
    \includegraphics[width=0.5\textwidth]{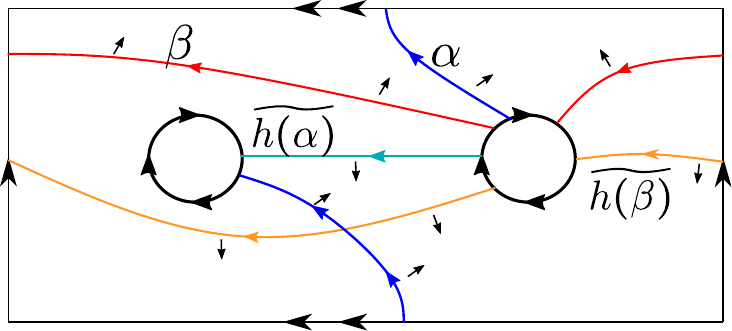}
    \caption{The sectors of $B$ are five half discs. The figure also describes the cusp directions of $B$ and it is easy to check that none of these half discs is sink.}
    \label{B_2 no sink}
\end{figure}
To prove that $B$ satisfies the hypotheses of Theorem \ref{boundary train tracks} we have to show that $\partial M_h\setminus \partial B$ is a union of bigons and that $B$ does not carry a torus.
By construction $\partial M_h\setminus \partial B$ is union of bigons (see Figure \ref{boundary B_2}) and since each sector of $B$ intersect $\partial M_h$ it follows that any surface carried by $B$ must intersect $\partial M_h$. Therefore $B$ does not carry any closed surfaces and in particular it does not carry tori.
\end{proof}

\begin{cor}\label{theorem part 2.2}
Let $\left(\sfrac{p_1}{q_1},\sfrac{p_2}{q_2}\right)$ be slopes in
$$
(0, \infty) \times (-1,1) \cup (-1,1)\times (0, \infty)
$$
Then the filled manifold $S^3_{\sfrac{p_1}{q_1}, \sfrac{p_2}{q_2}}({\rm WL})$ supports a coorientable taut foliation.
\end{cor}
\begin{proof}
By virtue of Theorem \ref{boundary train tracks} for any multislope $(s_1,s_2)$ realised by the boundary train tracks of $B$ there exists an essential lamination $\Lambda$ fully carried by $B$ intersecting the boundary of the exterior of the Whitehead link in parallel curves of multislope $(s_1,s_2)$. The boundary train tracks of $B$ are the following:
\begin{figure}[H]
   \centering
    \includegraphics[width=0.7\textwidth]{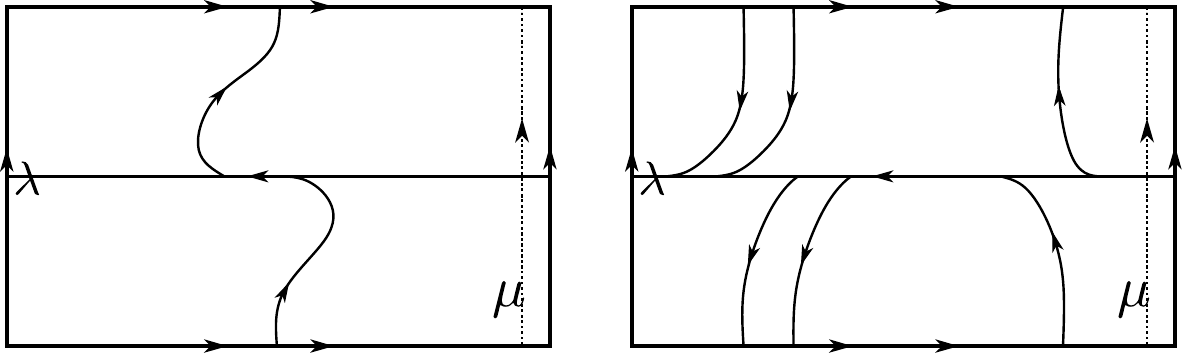}
    \caption{The boundary train tracks of $B$.}
    \label{boundary B_2}
\end{figure}

By assigning weights to these train tracks as in Figure \ref{boundary specifico pesi} it follows that these train tracks realise all the slopes in $(0,\infty)\times (-1,1)$.
\begin{figure}[H]
   \centering
    \includegraphics[width=0.7\textwidth]{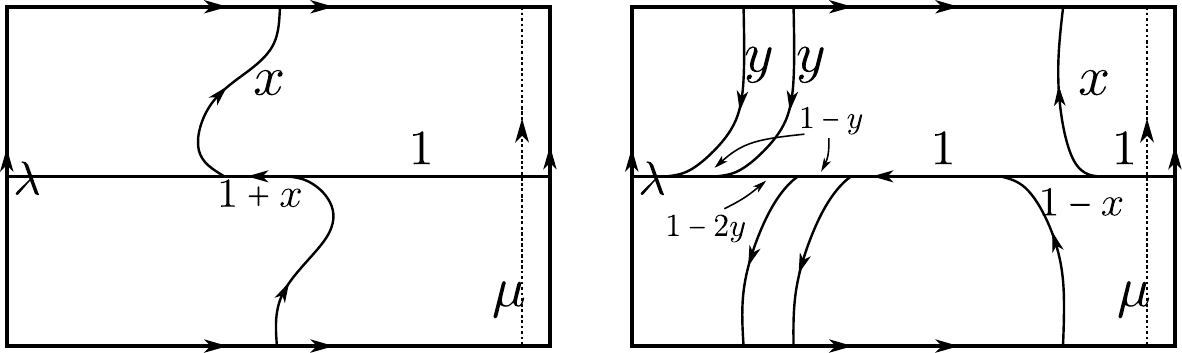}
    \caption{The boundary train tracks of $B$ with weight systems.}
    \label{boundary specifico pesi}
\end{figure}

Let $\Lambda$ be an essential lamination intersecting the boundary of $M_h$ in parallel curves of one of these multislopes $(s_1,s_2)$. Since each leaf of $\Lambda$ is carried by $B$ and each sector of $B$ intersects $\partial M_h$ it follows that all the leaves of $\Lambda$ intersect $\partial M_h$. It follows by the proof of Lemma \ref{extending laminations} that we can construct a foliation of $M_h$ such that each leaf of this foliation is parallel to some leaf of $\Lambda$. Therefore all the leaves of these foliation intersect $\partial M$ and as a consequence when we cap these leaves with the meridional discs of the solid tori we obtain a foliation of $M_h(s_1,s_2)$ with the property that the cores of the tori are transversals intersecting all the leaves.

We have proved that for any $\left(\sfrac{p_1}{q_1},\sfrac{p_2}{q_2}\right)\in (0,\infty)\times (-1,1)$ the manifold $S^3_{\sfrac{p_1}{q_1}, \sfrac{p_2}{q_2}}({\rm WL})$ supports a coorientable taut foliation. Since the Whitehead link is symmetric we deduce that the same result holds also for any $\left(\sfrac{p_1}{q_1},\sfrac{p_2}{q_2}\right)\in (-1,1)\times (0,\infty)$.
\end{proof}

\begin{proof}[Proof of Theorem \ref{theorem part 2}]
It is a consequence of Corollary \ref{corollario quasi tutto} and Corollary \ref{theorem part 2.2}.
\end{proof}

\section{Orderability}

In this last section we will prove Theorem \ref{theorem euler} and discuss some results about the orderability, and non-orderability, of some surgeries on the Whitehead link. We recall the following definition:

\begin{definition}
Let $G$ be a group. $G$ is \textbf{left orderable} if there exists a total order $<$ on $G$ that is invariant for the left multiplication by elements in $G$, i.e. such that for any $g,g'\in G$ we have that $g < g'$ if and only if $hg< hg'$ $\forall h \in G$.
\end{definition}

If $G$ is the fundamental group of a closed, orientable, irreducible $3$-manifold, left orderability translates in the following dynamical property. 

\begin{theorem}\cite{BRW}\label{characterization of orderability}
Let $N$ be a closed, irreducible, orientable $3$-manifold. Then $\pi_1(N)$ is left orderable if and only if there exists a non-trivial homomorphism $\varphi:\pi_1(N)\rightarrow {\rm Homeo}^+(\R)$.
\end{theorem}

This result yields us a theoretical way to connect taut foliations to left orderability in the following way. Suppose that $\mathcal{F}$ is a cooriented taut foliation on a rational homology $3$-sphere $N$. We can associate to $\mathcal{F}$ its tangent bundle $T\mathcal{F}$, that is a plane bundle over $N$. Being a plane bundle, we can associate to $T\mathcal{F}$ its Euler class $e(T\mathcal{F})\in H^2(N;\Z)$. Moreover, by a construction of Thurston (see \cite{CD}), it is possible to associate to $\mathcal{F}$ a non-trivial homomorphism 
$$
\varphi:\pi_1(N)\rightarrow {\rm Homeo}^+(S^1).
$$
Since there is an injective homomorphism from the universal cover $\widetilde{{\rm Homeo}^+(S^1)}$ of ${\rm Homeo}^+(S^1)$ into ${\rm Homeo}^+(\R)$, one would like to lift $\varphi$ to a homomorphism 
$$
\Tilde{\varphi}:\pi_1(N)\rightarrow \widetilde{{\rm Homeo}^+(S^1)}.
$$
The obstruction to find such a lift is again a cohomology class in $H^2(N;\Z)$, and it turns out that this class vanishes if and only if $e(T\mathcal{F})=0$. For more details we refer to \cite{BH}.

The upshot of the previous discussion is the following

\begin{theorem}\label{euler class}\cite{BH}
Let $N$ be a rational homology sphere and let $\mathcal{F}$ be a coorientable taut foliation on $N$. If the Euler class of $\mathcal{F}$ vanishes then $\pi_1(N)$ is left orderable.
\end{theorem}

We now consider the taut foliations obtained in the previous section and determine which of them have vanishing Euler class. To do this we will adapt part of the content of \cite{H} to our context.

We fix some notation. We denote with $M$ the exterior of {\rm WL} and we denote with $S$ the $2$-holed torus of Figure \ref{hopf plumbing} that is a Seifert surface for {\rm WL}. We fix a multislope $\left(\sfrac{p_1}{q_1},\sfrac{p_2}{q_2}\right)$, with $\sfrac{p_1}{q_1}<1$ or $\sfrac{p_2}{q_2}<1$, and we denote with $\mathcal{F}$ the foliation in $M$ intersecting $\partial M$ in parallel curves of multislope $\left(\sfrac{p_1}{q_1},\sfrac{p_2}{q_2}\right)$, as constructed in the proof of Theorem \ref{theorem part 2}. 
This foliation extends to a foliation $\hat{\mathcal{F}}$ of the filled manifold $S^3_{\sfrac{p_1}{q_1}, \sfrac{p_2}{q_2}}({\rm WL})$ so that in the glued solid tori $N_1$ and $N_2$ the foliation $\hat{\mathcal{F}}$ restricts to the standard foliations $\mathcal{D}_1$ and $\mathcal{D}_2$, which are the foliations by meridional discs. We can suppose without loss of generality that $p_1,p_2>0$. We orient the meridional disc $D_i$ of $N_i$ so that the gluing map identifies $\partial D_i$ with the oriented curve $p_i\mu_i +q_i\lambda_i$ in $\partial M$. 
\newline

The second homology group $H_2(M,\partial M; \Z)$ is isomorphic to $\Z^2$ and in particular we can fix as generators two properly embedded surfaces $S_1$ and $S_2$ that are duals to the meridians of the two components of the Whitehead link. Since the Whitehead link has linking number zero, these surfaces can be taken to be Seifert surfaces for the components of the link. In particular, these can be chosen to be tori with one disc removed, so that $\partial S_i=\lambda_i$. One of these tori is showed in Figure \ref{Toro_Whitehead} and the other can be obtained by an isotopy of $S^3$ exchanging the two components of {\rm WL}.

\begin{figure}[H]
   \centering
    \includegraphics[width=0.6\textwidth]{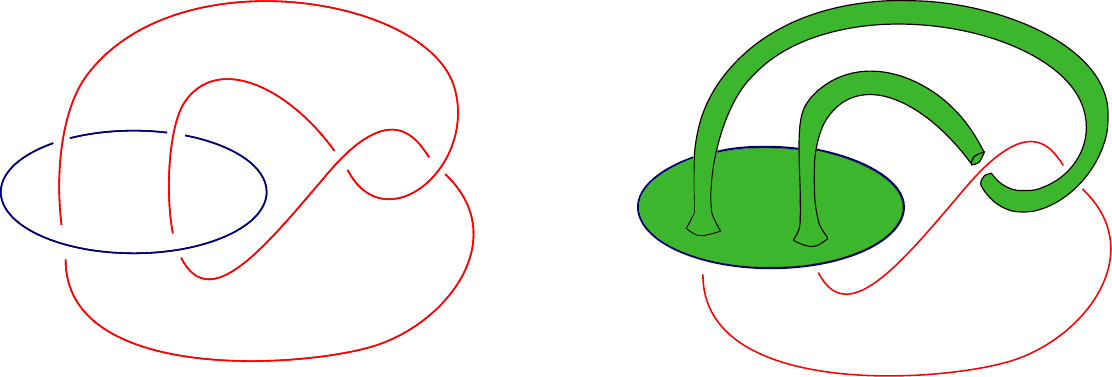}
    \caption{The $1$-holed torus depicted in this figure is one of the two generators of $H_2(M,\partial M;\Z)$.}
    \label{Toro_Whitehead}
\end{figure}

We fix a nowhere vanishing section $\sigma$ of $(T\mathcal{F})_{|\partial M}$ that is everywhere pointing outside of $M$. Hence the restrictions of $\sigma$ to the boundary components of $M$ also define nowhere vanishing sections $\sigma_i$ of $(T\mathcal{D}_i)_{|\partial N_i}$ everywhere pointing inside $N_i$, for $i=1,2$. These sections yield us relative Euler classes in $H^2(M,\partial M; \Z)$, $H^2(N_1,\partial N_1; \Z)$ and $H^2(N_2,\partial N_2; \Z)$, that we denote respectively with $e_{\sigma}(T\mathcal{F})$, $e_{\sigma_1}(T\mathcal{D}_1)$ and $e_{\sigma_2}(T\mathcal{D}_2).$ See \cite{H} for details.

Finally, we set $a_i=\langle e_{\sigma}(T\mathcal{F}),[S_i]\rangle$ and $b_i=\langle e_{\sigma_i}(T\mathcal{D}_i),[D_i]\rangle$, where $D_i$ is a meridional disc in $N_i$.

\begin{remark}
Notice that since $\mathcal{D}_i$ is the standard foliation of the solid torus by meridional discs, we have that $b_i$ coincides with $\pm\langle e_{\sigma_i}(TD_i),[D_i]\rangle =\pm \chi(D_i)=\pm 1$, where $TD_i$ denotes the tangent bundle of $D_i$ and where the sign depends on the orientation of the foliation $\mathcal{D}_i$.
\end{remark}

We are interested in knowing when $e(T\hat{\mathcal{F}})$ vanishes. The following proposition tells us exactly when this happens. Recall that without loss of generality we are supposing $p_1,p_2>0$, whereas the signs of $q_1$ and $q_2$ are arbitrary. 

\begin{prop}\label{orderability generale}
We have that $e(T\hat{\mathcal{F}})=0$ if and only if $e(T\mathcal{F})=0$ and $a_iq_i\equiv b_i  \pmod{p_i}$.
\end{prop}

\begin{proof}
The statement of this proposition is the generalisation to our case of the statements of \cite[Lemma 3.1]{H} and \cite[Theorem 1.4]{H} and the proof that is presented there adapts almost unaltered. We give a brief sketch of the proof and refer to \cite{H} for the details. In what follows the cohomology and homology groups are all implicitly assumed with integer coefficients and we will denote with $\overline{M}$ the $(\sfrac{p_1}{q_1},\sfrac{p_2}{q_2})$-surgery on $M$. Since $\overline{M}$ is a rational homology sphere as a consequence of the long exact sequence of the pair $(\overline{M},\partial M)$ we have

\begin{equation}
0\rightarrow H^1(\partial M)\overset{\delta}{\rightarrow} H^2(\overline{M},\partial M) \overset{\iota}{\rightarrow} H^2(\overline{M})\rightarrow 0. \label{relative cohomology}
\end{equation}
Moreover as a consequence of the Mayer-Vietoris sequence there is an isomorphism
\begin{equation}
H^2(M,\partial M)\oplus H^2(N_1,\partial N_1)\oplus  H^2(N_2,\partial N_2)\cong H^2(\overline{M},\partial M) \label{iso}
\end{equation}
defined by mapping the relative classes $(c_M,c_{N_1},c_{N_2})$  to the sum $\overline{c_M}+\overline{c_{N_1}}+\overline{c_{N_2}}$, where each of these cohomology classes is obtained by extending  to $\overline{M}$ the corresponding relative class by the zero map.

By using the identification given by the isomorphism in \eqref{iso} we obtain a short exact sequence:
\begin{equation}
0\rightarrow H^1(\partial M)\overset{\psi}{\rightarrow}H^2(M,\partial M)\oplus H^2(N_1,\partial N_1)\oplus  H^2(N_2,\partial N_2) \overset{\varphi}{\rightarrow} H^2(\overline{M})\rightarrow 0 \label{relative cohomology+ iso}
\end{equation}
where
$$
\psi(\beta)=(\delta_M \beta, (\delta_{N_1}\circ f^*_1)(\beta),(\delta_{N_2}\circ f^*_2)(\beta))
$$
with $f_1:\partial N_1\hookrightarrow \partial M$ and $f_2:\partial N_2\hookrightarrow \partial M$ denoting the gluing maps of the solid tori and with
$$
\delta_M:H^1(\partial M)\rightarrow H^2(M,\partial M)
$$
$$
\delta_{N_1}:H^1(\partial N_1)\rightarrow H^2(N_1,\partial N_1)$$ 
$$\delta_{N_2}:H^1(\partial N_2)\rightarrow H^2(N_2,\partial N_2)
$$
denoting the maps appearing in the long exact sequences of the pairs $(M,\partial M)$, $(N_1,\partial N_1)$ and $(N_2,\partial N_2)$.
\\

We suppose now that $e(T\hat{\mathcal{F}})=0$. By naturality of the Euler class, $e(T\mathcal{F})$ is the image of $e(T\hat{\mathcal{F}})=0$ under the map induced by the inclusion $M\hookrightarrow \overline{M}$ and therefore we have that $e(T\mathcal{F})=0$.

Moreover it also holds that
$$
\varphi(e_{\sigma}(T\mathcal{F}),e_{\sigma_1}(T\mathcal{D}_1),e_{\sigma_2}(T\mathcal{D}_2))=e(T\hat{\mathcal{F}})=0
$$
and therefore there exists $\beta\in H^1(\partial M)$ such that $\psi(\beta)=(e_{\sigma}(T\mathcal{F}),e_{\sigma_1}(T\mathcal{D}_1),e_{\sigma_2}(T\mathcal{D}_2))$; in other words $\beta$ satisfies
$$
\begin{cases}
\delta_M \beta=e_{\sigma}(T\mathcal{F})\\
(\delta_{N_1}\circ f^*_1)(\beta)=e_{\sigma_1}(T\mathcal{D}_1)\\
(\delta_{N_2}\circ f^*_2)(\beta)=e_{\sigma_2}(T\mathcal{D}_2)
\end{cases}.
$$
The following calculation verifies that $a_iq_i\equiv b_i  \pmod{p_i}$:
$$
b_i=\langle e_{\sigma_i}(T\mathcal{D}_i), [D_i]\rangle=\langle (\delta_{N_i}\circ f_i^*)(\beta), [D_i]\rangle=\langle \beta,[f_i(\partial D_i)] \rangle=
$$
$$
=\langle \beta,p_i\mu_i+q_i\lambda_i \rangle=p_i\langle \beta,\mu_i \rangle +q_i\langle \beta,\lambda_i \rangle=p_i\langle \beta,\mu_i \rangle+ q_ia_i
$$
where in the last equality we have used that 
$$
\langle \beta,\lambda_i\rangle=\langle \beta,[\partial S_i]\rangle=\langle\delta_M \beta, [S_i]\rangle=\langle e_{\sigma}(T\mathcal{F}),[S_i]\rangle=a_i.
$$
We now prove that if $e(T\mathcal{F})=0$ and $a_iq_i\equiv b_i  \pmod{p_i}$ for $i=1,2$, then $e(T\hat{\mathcal{F}})=0$.

We consider again the short exact sequence in \eqref{relative cohomology}. The nowhere vanishing section $\sigma$ defines an element $e_{\sigma}(T\hat{\mathcal{F}})\in H^2(\overline{M},\partial M)$ that satisfies $\iota(e_{\sigma}(T\hat{\mathcal{F}}))=e(T\hat{\mathcal{F}})$ and therefore if we prove that $e_{\sigma}(T\hat{\mathcal{F}})$ belongs to the image of $\delta: H^1(\partial M)\rightarrow H^2(\overline{M},\partial M)$ we obtain the thesis.
Morever under the isomorphism \eqref{iso} the element $(e_\sigma(T\mathcal{F}),e_{\sigma_1}(T\mathcal{D}_1),e_{\sigma_2}(T{\mathcal{D}_2}))$ corresponds to $e_{\sigma}(T\hat{\mathcal{F}})$ and therefore it is enough to prove that $(e_\sigma(T\mathcal{F}),e_{\sigma_1}(T\mathcal{D}_1),e_{\sigma_2}(T{\mathcal{D}_2}))$ belongs to the image of $\psi$ in the short exact sequence \eqref{relative cohomology+ iso}.

If we consider the long exact sequence of the pair $(M,\partial M)$ we have the following 
$$
H^1(M;\Z)\overset{\iota'_M}{\rightarrow} H^1(\partial M)\overset{\delta_M}{\rightarrow} H^2(M,\partial M)\overset{\iota''_M} \rightarrow H^2(M)
$$
and since $\iota''_M(e_{\sigma}(T\mathcal{F}))=e(T\mathcal{F})=0$ we deduce that there exists $\beta_0\in H^1(\partial M)$ such that $\delta_M(\beta_0)=e_{\sigma}(T\mathcal{F})\in H^2(M,\partial M)$. We now want to modify $\beta_0$ in order to find $\beta\in H^1(\partial M)$ that satisfies $$
\psi(\beta)=(e_\sigma(T\mathcal{F}),e_{\sigma_1}(T\mathcal{D}_1),e_{\sigma_2}(T{\mathcal{D}_2}))
$$
that is to say, such that
$$
\begin{cases}
\delta_M \beta=e_{\sigma}(T\mathcal{F})\\
(\delta_{N_1}\circ f^*_1)(\beta)=e_{\sigma_1}(T\mathcal{D}_1)\\
(\delta_{N_2}\circ f^*_2)(\beta)=e_{\sigma_2}(T\mathcal{D}_2).
\end{cases}
$$
We denote with $\mu_i^*\in H^1(\partial M)$ the dual of $\mu_i\in H_1(\partial M)$ and we define 
$$
\beta=\beta_0+n_1\mu_1^*+n_2\mu_2^* \quad \text{where }n_i=-\langle\beta_0,\mu_i\rangle-\frac{a_iq_i-b_i}{p_i}.
$$
Since $a_iq_i\equiv b_i  \pmod{p_i}$ for $i=1,2$ it follows that $n_i$ is an integer. Moreover, since $\beta-\beta_0\in \iota'_M(H^1(M))$ we have that $\delta_M(\beta_0)=\delta_M(\beta)=e_{\sigma}(T\mathcal{F})\in H^2(M,\partial M)$. We have to prove that $(\delta_{N_i}\circ f^*_i)(\beta)=e_{\sigma_i}(T\mathcal{D}_i)$ for $i=1,2$. Since 
$$
H^2(N_i,\partial N_i)\cong {\rm Hom}(H_2(N_i,\partial N_i,\Z))
$$
it is enough to prove that $\langle (\delta_{N_i}\circ f^*_i)(\beta),[D_i]\rangle=\langle e_{\sigma_i}(T\mathcal{D}_i), [D_i]\rangle$ and this is a consequence of the following computation (the case $i=2$ is analogous). 
$$
\langle (\delta_{N_1}\circ f^*_1)(\beta),[D_1]\rangle=\langle \beta, f_1(\partial D_1)\rangle=
$$
$$
=\langle \beta_0, p_1\mu_1+q_1\lambda_1\rangle+n_1\langle \mu_1^*,p_1\mu_1+q_1\lambda_1\rangle + n_2\langle \mu_2^*,p_1\mu_1+q_1\lambda_1\rangle=
$$
$$
=p_1\langle \beta_0,\mu_1\rangle+a_1q_1+p_1 \left(-\langle\beta_0,\mu_1\rangle-\frac{a_1q_1-b_1}{p_1}\right)=b_1=\langle e_{\sigma_1}(T\mathcal{D}_1), [D_1]\rangle
$$
where in the last line we have used again that
$$
\langle \beta_0,\lambda_1\rangle=\langle \beta_0,[\partial S_1]\rangle=\langle\delta_M \beta_0, [S_1]\rangle=\langle e_{\sigma}(T\mathcal{F}),[S_1]\rangle=a_1.
$$ 
and that $\langle\mu_2^*,\mu_1\rangle=\langle\mu_2^*,\lambda_1\rangle=0$
\end{proof}

\begin{namedtheorem}[\ref{theorem euler}]
Let $S^3_{\sfrac{p_1}{q_1}, \sfrac{p_2}{q_2}}({\rm WL})$ be the $\left(\sfrac{p_1}{q_1}, \sfrac{p_2}{q_2}\right)$-surgery on the Whitehead link, with $q_1,q_2\ne 0$ and $p_1,p_2>0$.

Then the taut foliations constructed in the proof of the Theorem \ref{theorem} have vanishing Euler class if and only if $|q_i|\equiv 1 \pmod{p_i}$ for $i=1,2.$

In particular, for all these manifolds the L-space conjecture holds.
\end{namedtheorem}
\begin{proof}
First of all we prove that $e(T\mathcal{F})=0$. In fact, let $T$ denote one of the boundary components of $M$; the inclusion of $T$ in $M$ induces an isomorphism $\iota:H^2(M;\Z)\rightarrow H^2(T; \Z)$. Therefore we have
$$
e(T\mathcal{F})=0 \Leftrightarrow \iota(e(T\mathcal{F}))=0.
$$
By naturality of the Euler class we have that 
$$
\iota(e(T\mathcal{F}))=e(T(\mathcal{F}_{|T}))
$$
and since $\mathcal{F}_{|T}$ admits a nowhere vanishing section we have that the last quantity is zero.

We now want to compute the numbers $a_i$. As a consequence of the proof of Theorem $1.7$ in \cite{H} we have that
\begin{itemize}
    \item $b_i=1$ if $q_i<0$, for $i=1,2$;
    \item $b_i=-1$ if $q_i>0$, for $i=1,2$;
    \item $\langle e_{\sigma}(T\mathcal{F}),[S]\rangle =\chi(S)=-2$.
\end{itemize}
 Since by construction $S$ intersects positively in one point the meridians of the components of the Whitehead link, we have the equality  $[S]=[S_1+S_2]$ in $H_2(M,\partial M;\Z)$ and hence
$$
a_1+a_2=\langle e_{\sigma}(T\mathcal{F}),[S_1]\rangle +\langle e_{\sigma}(T\mathcal{F}),[S_2]\rangle =\chi(S)=-2.
$$

As a consequence of \cite[Corollary~1, p.~118]{Th} for any $[F]\in H_2(M,\partial M;\Z)$ we have the inequality
$$
|\langle e_{\sigma}(T\mathcal{F}),[F]\rangle |\leq |\chi(F)|
$$
and since $S_1$ and $S_2$ are $1$-holed tori, this implies that $a_i=\langle e_{\sigma}(T\mathcal{F}),[S_i]\rangle =-1$ for $i=1,2$.
Therefore, by virtue of Proposition \ref{orderability generale} we have $e(T\hat{\mathcal{F}})=0$ if and only if for each $i=1,2$ it holds one of the following:
\begin{itemize}
    \item $q_i$ is positive and $q_i\equiv 1\quad \pmod{p_i}$;
    \item $q_i$ is negative and $q_i\equiv-1\quad \pmod{p_i}$.
\end{itemize}

In other words $e(T\hat{\mathcal{F}})=0$ if and only if 
$$
|q_i|\equiv 1 \quad \pmod{p_i} \quad\text{for } i=1,2
$$
that is exactly what we wanted.
\end{proof}
We point out the following straightforward consequence of Theorem \ref{theorem euler}.
\begin{cor}\label{corollary negative integers}
Let $d_1,d_2$ be two integers such that $d_1<0$ or $d_2<0$. Then the manifold $S^3_{d_1,d_2}({\rm WL})$ satisfies the L-space conjecture.\hfill{$\square$}
\end{cor}

We conclude by collecting from the literature some results regarding the orderability (or non-orderability) of some surgeries on the Whitehead link, obtaining a generalisation of Corollary \ref{corollary negative integers}.

\begin{namedtheorem}[\ref{Z+RS}]
Let $m\ne0$ be an integer.
\begin{itemize}
    \item If $m \leq-1$ then the manifolds $S^3_{m,\sfrac{p}{q}}({\rm WL})$ and $S^3_{\sfrac{p}{q},m}({\rm WL})$  have left orderable fundamental group for all rationals $\sfrac{p}{q}$.
    \item If $m \geq 1$ then the manifolds $S^3_{m,\sfrac{p}{q}}({\rm WL})$ and $S^3_{\sfrac{p}{q},m}({\rm WL})$ have non left orderable fundamental for all rationals $\sfrac{p}{q}\geq 1$.
\end{itemize}
In particular, all the rational homology spheres obtained by integer surgery on {\rm WL} satisfy the L-space conjecture.
\end{namedtheorem}

\begin{proof}
The manifold $S^3_{m,\bullet}({\rm WL})$ fibers over the circle if and only if $m$ is an integer (see \cite{HMW}). Moreover, in this case the fiber is a punctured torus.
\begin{itemize}
    \item When $m\leq -1$ the monodromy of $S^3_{m,\bullet}({\rm WL})$ can be extended to an Anosov diffeomorphism $\phi$ of the torus that preserves the orientations of its stable and unstable foliations, see \cite{HMW}.  The manifold $S^3_{m,\sfrac{p}{q}}({\rm WL})$ can be obtained by surgery along a closed orbit of $\phi$ in the mapping torus $M_{\phi}$ and as a consequence of \cite[Theorem ~1]{Z} we have that all the non-trivial fillings of $S^3_{m,\bullet}({\rm WL})$ have left orderable fundamental group. Since ${\rm WL}$ is symmetric, the same result holds for $S^3_{\bullet, m}({\rm WL})$.
    \item When $m\geq 3$ and $\sfrac{p}{q}\geq 1$, the fundamental group of the manifold $S^3_{m,\sfrac{p}{q}}({\rm WL})$ was studied by Roberts, Shareshian and Stein in \cite[Proposition 3.1]{RSS}, where they prove that it is not orderable. The technical details of their result are contained in the proofs of \cite[Lemma ~3.5]{RSS} and \cite[Corollary ~3.6]{RSS} and these also work in the case $m=2$ (notice that in their notation, this is the case ``$m=0$'').
    When $m=1$ the manifold $S^3_{1,\bullet}({\rm WL})$ is the exterior of the right-handed trefoil knot and its surgeries are well known \cite{Moser}. In fact all surgeries but one yield Seifert fibered manifolds and we have already proved that these are $L$-spaces; since the $L$-space conjecture holds for Seifert fibered manifolds we deduce that these surgeries are all non-orderable. The remaining surgery is a connected sum of two lens spaces and since its fundamental group has torsion it is not orderable as well.
\end{itemize}
This concludes the proof.
\end{proof}

\begin{remark}
Notice that even if the content of Corollary \ref{corollary negative integers} is generalised by Theorem \ref{Z+RS}, the statement of Theorem \ref{theorem euler} is not; in fact there are also non-integer rationals $\sfrac{p_1}{q_1}$ and $\sfrac{p_2}{q_2}$ that satisfy the hypotheses of Theorem \ref{theorem euler}.  
\end{remark}

In \cite{D} Dunfield considers a census of more than 300,000 hyperbolic rational homology spheres, testing the conjecture for this census. These manifolds are obtained by filling $1$-cusped hyperbolic $3$-manifolds that can be triangulated with at most $9$ ideal tetrahedra, see \cite{Burton}. We checked whether some of these manifolds studied by Dunfield arise as Dehn surgery on the Whitehead link and obtained the following

\begin{prop}\label{prop dunfield}
Among the $307,301$ rational homology spheres studied in \cite{D} at least $625$ are obtained as Dehn surgery on the Whitehead link. In \cite{D} it is proved that:
\begin{itemize}
    \item $300$ of these manifolds are orderable;
    \item $250$ are non-orderable.
\end{itemize} It follows from Theorem \ref{Z+RS} that $16$ of the remaining $75$ manifolds are orderable and $10$ are non-orderable.
\end{prop}
The code of the program can be found at \cite{Code}.
The surgery coefficients yielding these manifolds are plotted in Figure \ref{plot}.

\begin{figure}[H]
   \centering
    \includegraphics[width=0.6\textwidth]{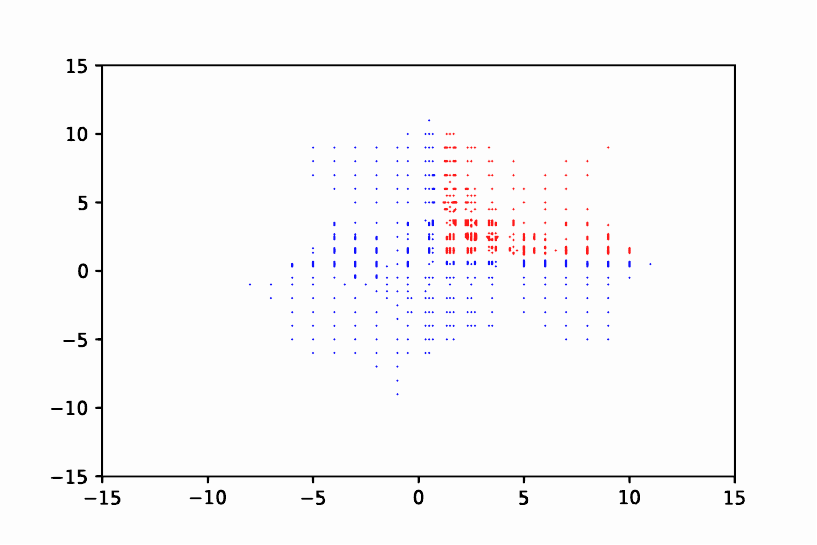}
    \caption{ In this figure the red dots represent the coefficients whose corresponding surgery is non-orderable and the blue dots represent the coefficients whose corresponding surgery is orderable.}
    \label{plot}
\end{figure}

The examples of Theorem \ref{Z+RS} and Proposition \ref{prop dunfield} are consistent with the conjecture, and together with Theorem \ref{theorem} confirm that the L-space conjecture holds also for all these manifolds.

\bibliographystyle{alpha}
\bibliography{main}

\textsc{Scuola Normale Superiore, Piazza dei Cavalieri 7, Pisa}\\
Email address: \texttt{diego.santoro95@gmail.com}

\end{document}